\documentclass[a4paper,10pt]{article}
\usepackage[utf8]{inputenc}

\usepackage[T1]{fontenc}
\usepackage[english]{babel}

\usepackage{authblk}

\usepackage[margin=1.3in]{geometry}

\usepackage{enumitem}

\usepackage[colorlinks=true,linkcolor=blue,citecolor=red,urlcolor=blue,pdfborder={0 0 0},pagebackref=true]{hyperref}

\usepackage{amsmath}
\usepackage{amssymb}
\usepackage{mathrsfs}
\usepackage{amsthm}
\usepackage{stmaryrd}      
\usepackage{MnSymbol}      
\usepackage{bbm}

\theoremstyle{plain}
\newtheorem{thm}{Theorem}
\newtheorem{prop}{Proposition}
\newtheorem{lem}{Lemma}
\newtheorem{coroll}{Corollary}
\newtheorem{claim}{Claim}

\theoremstyle{definition}
\newtheorem{Def}{Definition}

\theoremstyle{remark}
\newtheorem{rk}{Remark}

\newcommand{\bN}{\mathbb{N}}
\newcommand{\bZ}{\mathbb{Z}}
\newcommand{\bQ}{\mathbb{Q}}
\newcommand{\bR}{\mathbb{R}}

\newcommand{\bT}{\mathbb{T}}

\newcommand{\cA}{\mathcal{A}}

\newcommand{\cP}{\mathcal{P}}

\newcommand{\cR}{\mathcal{R}}

\newcommand{\tlambda}{\tilde{\lambda}}
\newcommand{\tgamma}{\tilde{\gamma}}
\newcommand{\tC}{\tilde{C}}

\newcommand{\abs}[1]{\left|#1\right|}

\newcommand{\norm}[1]{\left \|#1 \right\|}

\newcommand{\innerprod}[2]{\left\langle #1 \: , \: #2 \right\rangle}

\newcommand{\TV}[1]{\left \|#1 \right\|_{\mathrm{TV}}}
\newcommand{\supp}{\mathrm{supp} \:}

\newcommand{\Zn}{\bZ / n \bZ}
\newcommand{\GL}{\mathrm{GL}}
\newcommand{\SL}{\mathrm{SL}}

\newcommand{\hmu}{\hat{\mu}}

\newcommand{\AT}{A^{\top}}

\newcommand{\Wseps}{W^{s}_{\epsilon}}
\newcommand{\Wueps}{W^{u}_{\epsilon}}
\newcommand{\ec}{\epsilon_{c}}

\title{Accelerating Abelian Random Walks with Hyperbolic Dynamics}
\author[1,2,3]{Bastien Dubail\thanks{Correspondence to be sent to: bastien.dubail@inria.fr}}
\author[1,2]{Laurent Massoulié}
\date{}

\affil[1]{Département d’informatique de l’ENS, École normale supérieure, CNRS, PSL Research University, Paris, France}
\affil[2]{INRIA, Paris, France}
\affil[3]{Aix-Marseille Université, CNRS, I2M, Marseille, France}

\begin{document}

\maketitle

\begin{abstract}
	Given integers $d \geq 2, n \geq 1$, we consider affine random walks on torii $(\Zn)^{d}$ defined as $X_{t+1} = A X_{t} + B_{t} \mod n$, where $A \in \GL_{d}(\bZ)$ is a invertible matrix with integer entries and $(B_{t})_{t \geq 0}$ is a sequence of iid random increments on $\bZ^{d}$. We show that when $A$ has no eigenvalues of modulus $1$, this random walk mixes in $O(\log n \log \log n)$ steps as $n \rightarrow \infty$, and mixes actually in $O(\log n)$ steps only for almost all $n$. These results are similar to those of \cite{chung1987random} on the so-called Chung-Diaconis-Graham process, which corresponds to the case $d=1$. Our proof is based on the initial arguments of Chung, Diaconis and Graham, and relies extensively on the properties of the dynamical system $x \mapsto A^{\top} x$ on the continuous torus $\bR^{d} / \bZ^{d}$. Having no eigenvalue of modulus one makes this dynamical system a hyperbolic toral automorphism, a typical example of a chaotic system known to have a rich behaviour. As such our proof sheds new light on the speed-up gained by applying a deterministic map to a Markov chain.
\end{abstract}

\section{Introduction, main results}

In \cite{chung1987random}, Chung Diaconis and Graham investigated the behaviour of a Markov chain on $\Zn$ defined by $X_{t+1} = a X_{t} + B_t$, with $(B_t)$ a sequence of iid increments distributed on $\{-1, 0, 1 \}$. They proved for $a = 2$ that after $t= O(\log n \log \log n)$ steps the distribution of $X_t$ is close to uniform, thus showing a dramatic speed-up over the simple random walk, which needs $\Omega(n^2)$ steps.
Recently, attention has been brought back to the potential speed-up obtained by applying deterministic functions to Markov chains \cite{chatterjee2020speeding, he2020markov, he2021mixing, benhamou2021cutoff}. In this work, we study an analog of the Chung-Diaconis-Graham process in the multi-dimensional case. This case has been previously studied \cite{asci2001generating,hildebrand2008generating,asci2009generating,klyachko2020random} however the bounds proved in these papers did not match those of dimension one, or only in specific cases. This work improves these results and aims to provide a new look on the speed-up caused by the doubling map.

\bigskip

Let $d \geq 2$ be an integer, $\mu$ a finitely supported probability measure on $\bZ^{d}$. We write $\GL_{d}(\bZ)$ for the set of invertible matrices with integer entries and integer-valued inverse. A matrix is hyperbolic if it has no eigenvalue of modulus $1$.
Given $n \geq 1$ and $A \in \GL_{d}(\bZ)$ an invertible hyperbolic matrix, consider the discrete-time random walk $(X_t)_{t \geq 0}$ on $(\Zn)^{d} $ defined by
\begin{equation}\label{eq:def_RW}
	X_0 := 0 \qquad X_{t} = A X_{t-1} + B_t \mod n \quad \forall t \geq 1
\end{equation}
where $B_{t}$ are iid random variables on $\bZ^{d}$ with distribution $\mu$. Such random walks have been first considered by Chung, Diaconis and Graham in \cite{chung1987random} in the case $d =1$, $A =2$.

$A$ induces a bijection on the finite set $(\Zn)^{d}$. It is easy to see from this that the uniform measure $U$ on $(\Zn)^{d}$ is invariant for $X_t$, so $X_t$ will converge in law to $U$ provided it is irreducible and aperiodic. Convergence to stationarity of finite Markov chains is classically measured by total variation distance: for any pair of measures $p, q$ on a finite set $X$, one sets
\[ 
 \TV{p - q} := \max_{A \subseteq X} \abs{p(A) - q(A)} = \frac{1}{2} \sum_{x \in X} \abs{p(x) - q(x)}.
\]
The mixing time is then defined as the time needed for a Markov chain to get at distance at most $\epsilon$ to its stationary distribution, for a fixed parameter $\epsilon$.

Let $\supp \mu$ denote the support of $\mu$ and consider $H$ the smallest $A$-invariant subgroup of $\bZ^{d}$ that contains $\supp \mu - \supp \mu := \{ x - y: x ,y \in \supp \mu\}$. 

As a subgroup of $\bZ^{d}$, $H$ is itself isomorphic to $\bZ^{k}$ for some integer $k \leq d$; see for instance \cite[Thm 7.8]{lang2002algebra}. More precisely, there exists a basis of $\bZ^{d}$, that is a family $(u_i)_{i=1}^{d}$ that generates $\bZ^{d}$ as a group, and positive integers $(a_i)_{i=1}^{k}$, such that $H$ is the subgroup generated by $(a_i u_i)_{i=1}^{k}$. The integer $\mathrm{rk} \, H := k$ is called the rank of $H$ and is equal to the dimension of the $\bQ$-subvector space of $\bQ^{d}$ spanned by $H$. Our first result gives an upper bound on the mixing time of the same order as the one given by Chung, Diaconis and Graham in \cite{chung1987random}. It applies to any dimension $d \geq 2$ and improve the results of \cite{asci2009generating,hildebrand2008generating}. Throughout the paper, $\log$ denotes the natural logarithm.

\begin{thm}\label{thm:upperbound}
	Let $d \geq 2$ be an integer, $\mu$ a probability measure on $\bZ^d$ and $A \in \GL_{d}(\bZ)$. Let $(B_t)_{t \geq 1}$ be iid random variables with law $\mu$ and consider the random walk $(X_t)_{t \geq 0}$ on $(\Zn)^{d}$ defined by \eqref{eq:def_RW}.
	
	Let $P^t(0, \cdot)$ denote the distribution of $X_t$ and $U$ the uniform measure on $(\Zn)^{d}$. Let $H$ be the smallest $A$-invariant subgroup of $\bZ^{d}$ that contains $\supp \mu - \supp \mu$. Suppose $H$ is of rank $d$, generated by $(a_i u_i)_{i=1}^{d}$ for some basis $(u_i)_{i=1}^{d}$ of $\bZ^{d}$ and positive integers $(a_i)_{i=1}^{d}$. If $A$ is hyperbolic, there exists a constant $C > 0$ such that, for all $n$ coprime with all the $a_i$, if $t > C \log n \log \log n$ then $\TV{P^{t}(0, \cdot) - U} \rightarrow 0$ as $n \rightarrow \infty$.
\end{thm}

\begin{rk}\label{rk:hyperbolic}
	The hyperbolicity assumption on $A$ is essential. Theorem 3.11 of \cite{asci2009generating} shows that if $A$ has an eigenvalue which is a root of unity, $\Omega(n^{2})$ steps are necessary.
	The proof of Theorem \ref{thm:upperbound} will use extensively the properties of the deterministic dynamical system $x \mapsto A^{\top} x$ on the continuous torus $\bR^{d} / \bZ^{d}$. Such dynamical systems are commonly refered to as hyperbolic toral automorphisms and are typical examples of dynamical systems exhibiting chaotic properties. Thus Theorem \ref{thm:upperbound} and its proof suggest that the speed-up observed is the consequence of specific features of hyperbolic dynamical systems. 
\end{rk}

\begin{rk}
	On the other hand, the invertiblity assumption on $A$ seems superfluous. We conjecture it suffices to have $\det A \neq 0$ and restrict to $n$ coprime with $\det A$, as it is already the case for dimension $1$ \cite{chung1987random,hildebrand1993, hildebrand1996random}. To prove Theorem \ref{thm:upperbound} we make use of Markov partitions, which apply essentially to invertible maps. Many arguments in this paper, even those about Markov partitions, only use "forward trajectories" and could thus be extended to the non-invertible case. The main issue is thus to show existence of Markov partitions for non invertible maps. 
\end{rk}

\begin{rk}
	As one can guess, the condition given on $H$ and $n$ is necessary to have irreducibility and aperiodicity of the random walk. This will be proved subsequently in Proposition \ref{prop:cns_aper}. Thus Theorem \ref{thm:upperbound} shows that for all $n$, $O(\log n \log \log n)$ steps are sufficient to reach stationarity as long as convergence holds.
\end{rk}

\begin{rk}
	Let us remark also that there is no loss of generality in supposing the walk is started at $0$. Indeed, should the random walk start at $x \in (\Zn)^{d}$, we can always write $X_t = A^{t} x + \tilde{X}_t$ with $\tilde{X}_t$ a random walk started at $0$. Since $A^{t} x$ is deterministic, it does not affect total variation distance.
\end{rk}

Our second result establishes a lower bound on the mixing time, under a condition of finite entropy. Given any measure $p$ on a discrete set $X$, we recall the entropy of $p$ is defined as
\[ 
	H(p) := - \sum_{x \in X} p_x \log p_x.
\]

\begin{thm}\label{thm:lowerbound}
	Let $\mu$ be a measure on $\bZ^{d}$ with finite entropy. Let $(X_t)_{t \geq 0}$ be as in \eqref{eq:def_RW}. Then 
	\begin{equation}
	\TV{P^{t}(0, \cdot) - U} \geq 1- \frac{t H(\mu) + \log 2}{\log n}.
	\end{equation}
\end{thm}

The upper bound of Theorem \ref{thm:upperbound} is valid for all $n$, as long as convergence to uniformity holds. As it was already the case for the Chung-Diaconis-Graham process in dimension one, it turns out that for almost all $n$ tending to infinity only $O(\log n)$ steps suffice to reach stationarity. A property $\cP$ is satisfied by almost all integer if the proportion of integers smaller than $n$ that satisfy $\cP$ tends to $1$ as $n \rightarrow \infty$.

\begin{thm}\label{thm:almostall}
	Consider the setting of Theorem \ref{thm:upperbound}. There exist a constant $C$ such that for almost all $n$, if $t > C \log n $ then $\TV{P^{t}(0, \cdot) - U} \rightarrow 0$ as $n \rightarrow \infty$.
\end{thm}

Affine random walks defined by \eqref{eq:def_RW} allow in particular to study random walks defined by recursions of higher order. For instance the so-called Fibonacci random walk on $\Zn$, $X_{t+1} = X_{t} + X_{t-1} + B_{t+1}$, studied in \cite{chatterjee2020speeding}, can naturally be written as an affine random walk, using the matrix $A= \left( \begin{smallmatrix}
	1 & 1 \\ 1 & 0
\end{smallmatrix}\right)$, which is hyperbolic. Theorems \ref{thm:upperbound} and \ref{thm:almostall} thus imply that when convergence holds, $O(\log n \log \log n)$ or $O(\log n)$ steps are sufficient to reach stationarity. This applies in particular to the case where $B_t$ is uniform in $\{-1, 0 , 1 \}$ and improves the result of \cite{chatterjee2020speeding}.

\begin{rk}
	From Theorem \ref{thm:almostall}, one may inquire if the upper bound of $O(\log n)$ is valid for all $n$. We conjecture it is not. In dimension $d=1$, $A=2$ and $B_t$ uniform in $\{0, \pm 1\}$, \cite{chung1987random} gives specific values of $n_k$, including $n_k = 2^{k} -1$, for which $O(\log n_k \log \log n_k)$ is in fact the right order of magnitude. This phenomenon occurs because for such $n_k$, the map induced by $x \mapsto 2x$ on $\bZ / n \bZ$ has order $O(\log n_k)$. Thus we expect that for certain distributions of $B_t$, $O(\log n \log \log n)$ is indeed the right mixing time whenever the permutation induced by the matrix $A$ has order $O(\log n)$. This could be an interesting question for future research.
\end{rk}

\subsection{Relations with previous work} 

On the Chung-Diaconis-Graham process itself, ie the case $d = 1$, great effort has been made to improve the bounds on the mixing rates, see \cite{hildebrand1996random,hildebrand2006chung,hildebrand2009lower,hildebrand2019lower,neville2011lower}. The latest paper is \cite{eberhard2020mixing}, where Eberhard and Varjú eventually established the cutoff phenomenon (sharp transition to stationarity) at time $c \log n$ for almost all $n$ and some 
explicit constant $c$. 

There has also been interest for generalizations of this process: Hildebrand  considered the case $X_{t} = a_{t} X_{t-1}  + B_{t}$ where a new random multiplier $a_t$ is applied at each step \cite{hildebrand1993,hildebrand2021multiplicatively}, while Asci  \cite{asci2001generating,asci2009generating} and Hildebrand, McCollum \cite{hildebrand2008generating} considered $d$-dimensional generalizations, obtaining upper bounds of order $O(\log n)^{2})$. Asci \cite{asci2009generating} and recently Klyachko \cite{klyachko2020random} managed to prove order $O(\log n \log \log n)$ for specific cases in dimension $2$. Our results improve on these works, giving a $O(\log n \log \log n)$ upper bound for any hyperbolic matrix $A \in \GL_{d}(\bZ)$. One crucial argument in the one-dimensional case was the use of the binary decomposition in order to identify the dynamical system $x \mapsto 2 x$ on $\bR / \bZ$ with a shift, which a priori had no direct extension to the multi-dimensional case, thus accounting for the difference in the order of magnitude. Klyachko \cite{klyachko2020random} uses a similar technique, $\beta$-ary expansions, to get the $\log n \log \log n$ bound. This paper uses the theory of Markov partitions which provides a general framework where dynamical system can be identified to shifts. Both kind of expansions are related.

Another related result is \cite{diaconis1992affine}, where Diaconis and Graham investigated the case of an affine random walk on the hypercube. With our notations this would amount to taking $n =2$ and $d \rightarrow \infty$, which is not a case covered here. 

On the other hand, the effect of applying determistic bijections on the mixing times of Markov chains has been recently reconsidered from a general point of view. In \cite{bordenave2020spectral}, Bordenave, Qiu and Zhang  prove an upper bound on the second eigenvalue of the product of a bistochastic matrix with a uniform permutation. In \cite{chatterjee2020speeding} Diaconis and Chatterjee consider a generic Markov chain on $n$ states, and give a geometric condition on deterministic bijections which yields a mixing time of order $O(\log n)$. This condition is made to increase the expansion of the Markov chain to $\Omega(1)$, a property known to imply mixing in $O( \log n )$. Moreover, it holds with large probability when the bijection is chosen uniformly at random, so almost every bijection actually yields fast mixing. Those results were subsequently improved by Ben Hamou and Peres in \cite{benhamou2021cutoff}, who established the cutoff phenomenon in the case the bijection is picked uniformly at random. 

Although speed-up occurs with almost every bijection, it was noticed however in \cite{chatterjee2020speeding} that the doubling map $x \mapsto 2 x$ does not fall in that category, so expansion is not the reason accounting for the acceleration. Thus a second goal of our paper is to suggest an alternative explanation as to why the doubling map improves the mixing rate. As explained in Remark \ref{rk:hyperbolic}, at the heart of our proof is the behaviour of the hyperbolic dynamical system $x \mapsto A^{\top} x$. Hence the speed-up established here may be specific to such maps and of a different nature than the acceleration proved for almost all bijection.

Finally, let us mention two other examples of explicit deterministic bijections implying speed-ups that have been studied. On the cycle $\bZ / p \bZ$ with $p$ prime, He \cite{he2020markov} proves an almost linear mixing time in $p$ when the bijection considered is a rational function. This bound has been improved in \cite{he2021mixing} in the specific case of the inverse function $f : x \mapsto 1/x$ if $x \neq 0$, $f(0) := 0$, where He, Pham and Xu shows that $O(\log p)$ steps suffice to reach stationarity. 

\subsection{Discussion}

In the two following paragraphs we discuss motivations and related models.

\paragraph{Random matrices:}

One can inquire about applying random matrices instead of deterministic ones, ie consider random walks of the form $X_{t} = A_{t} X_{t-1} + B_{t}$, with for instance $(A_t)_{t \geq 1}$ iid random matrices. For $d =1$, such cases were considered in \cite{hildebrand1993,hildebrand2021multiplicatively}. 

The simplest randomness we can add is by taking $A_t = A^{\epsilon_t}$ where $\epsilon_t$ are iid Bernoulli random variables of parameter $1/2$. This may slow down the convergence to stationarity, for instance with $B_t$ corresponding to increments of the simple random walk. With constant $A$, acceleration comes from applying successive powers $A^t$ to the random increments $B_k$. Now with $A_t = A^{\epsilon_t}$ the exponent of $A$ is to be replaced with the $t$-th step a random walk on $\bZ$. Because of the diffusive behaviour of the random walk on $\bZ$, it is of order at most $O(\sqrt{t})$. As a consequence, the mixing time becomes $O( (\log n)^{2})$. This can be proven with the same argumentation as in \cite{hildebrand2021multiplicatively}, which treats the case $d =1$.

Another case which can be considered is when $A_t$ are iid with distribution $\nu$, such that the subgroup generated by $\supp \nu$ is $\SL_{d}(\bZ)$. In this case, it is possible to prove acceleration to $(O (\log n))$, without restriction on $n$, while removing the constraint that the matrices have no eigenvalues of modulus $1$. The idea is to consider the evolution of the pair $(X_t, S_t)$ with $S_t := A_1 \cdots A_t$. Since the $A_t$ are iid, the process $(X_t, S_t)$ is a convolution random walk, ie a product of iid increments, on the semi-direct product $(\bZ / n\bZ)^{d} \rtimes \mathrm{SL}_{d}(\Zn)$, whose group operation is defined by $(x, A) (y, B) = (x + Ay, AB)$. This random walk can then be obtained as the reduction modulo $n$ of a convolution random walk on the infinite group $\bZ^{d} \rtimes \mathrm{SL}_{d}(\bZ)$. Such groups have been extensively studied for their expansion properties, as they give rise to deterministic constructions of families of expander graphs, including quotients $\SL_{d} (\Zn))$. We refer to the monography of Tao \cite{tao2015expansion} for an introduction to this topic. 

Since expansion is related to the fast-mixing of simple random walks, the same tools can be used to prove fast mixing of affine random walks. Incidentally, such results have already been used in \cite{he2021mixing} to prove speed-up for the inverse mapping on $\bZ / p \bZ$. Let us give a few examples.

Let $S$ be a finite symmetric set of generators of $\bZ^{d} \rtimes \SL_{d}(\bZ)$ and consider the Cayley graph $G_n = (V_n, E_n)$ where $V_n = (\Zn )^{2}$ and $(x,y) \in E_n$ is an edge if and only if there exists $s=(z,A) \in S$ such that $y = sx := Ax + z$. Consider the simple random walk, which at each step is multiplied by a uniform element of $S$. To ensure aperiodicity we can always impose laziness, ie allow the walk to stay put with positive probability. Theorem 2.4.1 in \cite{tao2015expansion} proves that the graphs $G_n$ are expanders, which implies that the (lazy) simple random walk mixes in $O(\log n)$ steps. In particular this applies to the random walk on $(\Zn )^{2}$ which has holding probability $1/2$ and otherwise moves from $(x,y)$ to $(x \pm 1, y), (x, y \pm 1)$, $(x \pm y, y)$ or $(x, y \pm x)$ with probability $1/16$. 


Let us finally mention the result of \cite{lindenstrauss2016spectral}, which shows that the expansion of a simple random walk on a semi-direct product $(\bZ / p \bZ)^{d} \rtimes \SL_{d}(\bZ / p \bZ)$, $p$ prime, and that of the projection to $\SL_{d}(\bZ / p \bZ)$, are of the same order. Hence in this case the mixing properties of an affine random walk are essentially determined by the product of random matrices.

\paragraph{Interacting random walks:}

There has been another motivation to generalizing the result of \cite{chung1987random} to the high-dimensional case. Suppose one has two Markov chains $(X_t)$, $(Y_t)$ on the same finite state space. Would it be possible to make them "interact" in such a way as to accelerate the convergence to stationarity? A natural example of interaction is group multiplication. Thus one could investigate Markov chains $Z_t = (X_t, Y_t)$ on $G \times G$, with $G$ a finite group, defined for instance such that, with some probability we add to $Z_t$ a random increment, and with remaining probability we make its two coordinates interact, in the sense that either $X_{t+1} := Y_{t} X_{t}$ or $Y_{t+1} := Y_{t} X_{t}$ with some probability. Such models were proposed in \cite{chatterjee2020speeding}: given a finite set $S$, $P$ a transition matrix on $S$ and $f : S^{d} \rightarrow S$ a function such that $f(\cdot, x_2, \ldots, x_d)$ is a bijection for all $x_2, \ldots , x_d$, the authors inquire about the random walk that moves deterministically from a state $(X_1, \ldots X_d)$ to the state $(X_2, \ldots, X_{d-1}, f(X_1, \ldots X_d)) $ and then takes a step of $P$ in the last coordinate. The latter model allows in particular to study Markov chains defined by recursion of higher order, like $X_{t+1} = X_{t} X_{t-1} B_{t+1}$ with $(B_t)$ a sequence of i.i.d. random variables on the group. However at this stage there is no general result on the mixing rates of such Markov chains.

A case at hand is that of Abelian groups. When the group is Abelian, bijections such as $(x,y) \mapsto (y, x+y)$ are group homomorphisms. Moreover every finite Abelian group $G$ is isomorphic to a product of cyclic groups $G \simeq \prod_{i} \bZ / k_i \bZ$, and one can ensure uniqueness of the $k_i$ by imposing that $k_i$ divides $k_{i+1}$ for all $i$; see \cite[Thm 7.7]{lang2002algebra}. Then $G^{d} \simeq \prod_{i} (\bZ / k_i \bZ)^{d}$ and the previous homomorphisms are those obtained by taking a matrix $A \in GL_{d} (\bZ)$ and applying $A$ to each $i$-th coordinate. Thus the projection on an affine random walk on the $i$-th coordinate yields an affine random walk on the torus $(\bZ / k_{i} \bZ)^{d}$, to which we can apply Theorem \ref{thm:upperbound}.

What about non-linear interactions? The proof technique initiated in \cite{chung1987random} which has been central to all works on affine random walks of $(\Zn)^{d}$ relies heavily on linearity. The question therefore remains open. Let us remark though that when the bijection takes the form of random multiplications between coordinates in $G^{d}$, the model looks very similar to the so-called product replacement algorithm. The latter is designed to construct uniform sets of generators of a finite group. Given a $k$-uplet of generators of a group $G$, the algorithm consists at each step, to select two generators at random and multiply one by the other. Bounds of order $\log \abs{G}$ on the mixing time are available when $G$ is Abelian. In the general case this is conjectured to be the correct order, as it would be the consequence of Kazhdan's (T) property for the automorphism groups of free groups; see \cite{lubotzky2001product,pak2014we}. As far as we know this is still an open question.


\paragraph{Organization of the paper }
	We start by proving the lower bound in Section \ref{section:lower}. Section \ref{section:upper}, which forms the crux of the paper, contains the essential arguments used in the proof of Theorem \ref{thm:upperbound}, including results from hyperbolic dynamics. It is followed by Section \ref{section:irr_aper}, where we establish that the necessity of conditions in \ref{thm:upperbound} for irreducibility and aperiodicity. Then in Section \ref{section:almostall} we prove Theorem \ref{thm:almostall}. Finally, Section \ref{section:markov} provides the definitions and proofs of the results about hyperbolic dynamics that are needed in the other sections.

\section{A Lower Bound for all $n$: proof of Theorem \ref{thm:lowerbound}}\label{section:lower}

Entropic arguments to get lower bounds on mixing times have become standard. The following inequality, due to Fannes and Audenaert \cite{audenaert2007sharp}, was already used in \cite{he2021mixing}. An entropic argument was also used in \cite{eberhard2020mixing} to lower bound the mixing time and prove the cutoff phenomenon.

\begin{prop}[Fannes and Audenaert \cite{audenaert2007sharp}]\label{prop:FA}
	Let $\mu, \nu$ be two probability measures defined on a finite set of size $N \geq 1$. Let $\epsilon := \TV{\mu - \nu}$ denote their total variation distance. Then
	\begin{equation}
		\abs{H(\mu) - H(\nu)} \leq \epsilon \log (N-1) + H(\epsilon)
	\end{equation} 
	where $H(\epsilon) = - \epsilon \log \epsilon - (1 - \epsilon) \log (1 - \epsilon)$ is the entropy of a Bernoulli random variable of parameter $\epsilon$.
\end{prop}

\begin{proof}[Proof of Theorem \ref{thm:lowerbound}]
	We can bound the entropy of a Bernoulli random variable by $\log 2$, hence Proposition \ref{prop:FA} shows 
	\[ 
	\TV{P^{t}(0, \cdot) -U} \geq 1 - \frac{H(X_t) + \log 2}{\log n} .
	\]
	On the other hand the entropy of $X_t$ grows at most linearly as $t H(\mu)$. Indeed we can construct $X_t$ as $X_t = Y_t \mod n$, where $Y_t$ is an affine random walk on $\bZ^{d}$ defined by $Y_0 := 0, Y_{t} := A Y_{t-1} + B_{t}$ for all $t \geq 1$. Using the well known fact that entropy is non-increasing under deterministic transforms yields $H(X_t) \leq H(Y_t)$. Using conditional entropies, we can then notice $(H(Y_t))_{t \geq 0}$ forms a subadditive sequence, as
	\begin{align*}
		H(Y_{t+s}) &\leq H(Y_{t+s}, Y_t) = H(Y_t) + H(Y_{t+s} | Y_t) \\
		&= H(Y_t) + H (Y_s)
	\end{align*}
	To get the second line, notice that one can write $Y_{t+s} = A^{s} Y_t + Z_s$ with $Z_s$ a random variable independent from $Y_t$ and distributed like $Y_s$, so that $H(Y_{t+s} \ | \ Y_t) = H(Y_s)$.
	As a consequence of subadditivity, $H(Y_t) \leq t H(Y_1) = t H(\mu)$, which proves the result.
\end{proof}

\section{General upper bound: proof of Theorem \ref{thm:upperbound}}\label{section:upper}

\paragraph{Main idea of the proof.} The basic idea is to use Fourier analysis: by the classical upper bound lemma of Diaconis and Shahshahani (see \eqref{eq:upperboundlemma} below), the $\ell^{2}$-distance between the law $P^{t}:=P^{t}(0, \cdot)$ of the random walk and the uniform measure can be expressed in terms of $\widehat{P^{t}}$, the discrete Fourier transform of the law. However, by the linearity of the dynamics, the process $X_{t}$ is a sum of independent random variables, so its Fourier transform factorizes. The factors of $\widehat{P^{t}}$ are the Fourier transform $\hmu$ of $\mu$ composed with successive powers of $A^{\top}$, the transpose of the matrix $A$. The goal will thus be to get the best upper bound on $\abs{\hmu((A^{\top})^s \rho)}$ for all $\rho \in (\Zn)^{d}$, $\rho \neq 0$. Considering a function $f$, independent of $n$, defined on the continuous torus $\bR^{d} / \bZ^{d}$ such that $f(\rho / n) = \abs{\hmu(\rho)}$ for all $\rho \in \bT_{n}$, we will then identify a compact subset $W^{c}$ of $\bR^{d} / \bZ^{d}$ on which $f$ is bounded by a constant $\gamma < 1$. Thus we can bound $\widehat{P^{t}}(\rho)$ by $\gamma^{l}$, where $l$ is the number of steps $s \leq t$ for which $(A^{\top})^{s} \rho \in W^{c}$. For that matter, we will use the theory of hyperbolic toral automorphisms to code such orbits by bi-infinite sequences. By doing so, the map $x \mapsto A^{\top} x$ gets identified with a shift, which allows to bound the number of times the orbits cross $W^{c}$. 

\subsection{Fourier analysis on the torus}

Let $n \geq 1$ be an integer. From now on we write $\bT_n := (\Zn )^{d}$. Let $(X_t)_{t \geq 0}$ be defined as in \eqref{eq:def_RW}, and for all $t \geq 0$ write $P^{t}$ for the law of $X_t$.


Let $\innerprod{\cdot}{\cdot}$ be the usual inner product on $\bR^{d}$.  Given $\rho \in \bT_n$ and a function $f$ on $G$, the Fourier transform of $f$ at $\rho$ is 
\[ 
\hat{f}(\rho) := \sum_{x \in \bT_n} f(x) e^{2i \pi \innerprod{x}{\rho} / n}.
\] 
In this expression $x, \rho$ are implicitly identified with their representative in $[0,n-1]^{d}$. Due to the $2 \pi$-periodicity of the complex exponential, the result is independent of the choice of representatives. Such identifications will be done frequently in the sequel without further explicit mention. 

Bounds on mixing times can be obtained through Fourier transforms thanks to the upper bound lemma of Diaconis and Shahshahani \cite{diaconis1988group}. Recall $U$ denotes the uniform measure on $\bT_n$. In the present setup it yields that for any measure $\nu$ on $\bT_n$:
\begin{equation}\label{eq:upperboundlemma}	
	4 \TV{\nu - U} \leq \sum_{g \neq 0} \abs{\hat{\nu}(g)}^2.
\end{equation}

This bound will be applied to $P^t$. The linearity of the problem yields a nice factorization of the Fourier transform, as shown by the following lemma, initially due to Asci \cite{asci2001generating}:
\begin{lem}\label{lem:fourier_product}
	For all $t \geq 1$, for all $\rho \in \bT_n$, 
	\begin{equation}\label{eq:fourier_general}
		\widehat{P^{t}}(\rho) = \prod_{j=0}^{t-1} \hmu((A^{\top})^j \rho).
	\end{equation}
\end{lem}

\begin{proof}
	Applying the definition of the Fourier transform, for all $t \geq 1$,
	\begin{align*}
		\widehat{P^{t}}(\rho) &= \sum_{ \substack{x \in \bT_n \\ y \in \supp \mu}} P^{t-1}(x) \mu(y) e^{2 i \pi \innerprod{A x+y}{\rho} / n} \\
		&= \sum_{x \in \bT_n} P^{t-1}(x)  e^{2 i \pi \innerprod{x}{A^{\top} \rho} / n} \sum_{y \in \supp \mu} \mu(y) e^{2 i \pi \innerprod{y}{\rho} / n} \\
		&= \widehat{P^{t-1}}(A^{\top} \rho) \hmu(\rho)
	\end{align*}
	which gives the result by induction, since $\widehat{P^{0}} = 1$.
\end{proof}

\subsection{Expansiveness}\label{subsec:expansiveness}

Let $\bT^{d} := \bR^{d} / \bZ^{d}$. $A$ induces a bijective map on $\bT^d$, which is identified with $A$. 

The factorization of the Fourier transform provided by Lemma \ref{lem:fourier_product} motivates the study of the deterministic dynamical system $x \mapsto A^{\top} x$ on $\bT^{d}$. In the next section, we will see that by partitioning $\bT^{d}$ into a finite number of "rectangles", we can represent any point by the sequence of rectangles its trajectory passes through. A key ingredient to make this representation work is the expansiveness property: the orbits of two distinct points must separate from each other at exponential rate. This in turn will ensure that provided rectangles of the partition are small enough, two distinct points cannot have their trajectories going through the exact same set of rectangles, hence representations are well-defined. 
Note that from Lemma \ref{lem:fourier_product} the results of the two next sections will be applied later on with $A^{\top}$.

\bigskip

Let us introduce the quotient metric on $\bT^{d}$.
Let $p$ be the natural projection of $\bR^{d}$ onto $\bT^{d}$. Given a norm $\norm{\cdot}$ on $\bR^{d}$, the quotient metric $d$, defined as
\[ 
d(x,y) = \inf \left\{ \norm{\hat{x} - \hat{y}}, \hat{x} \in p^{-1}(x), \hat{y} \in p^{-1}(y) \right\}
\]
makes $\bT^{d}$ a compact metric space.

Since $A$ has no eigenvalue of modulus one, by factorizing the characteristic polynomial of $A$ into a product of irreducible factors and regrouping them, one can write it as the product of two real polynomials $P_u, P_s$, the eigenvalues of which all have modulus strictly above or below $1$ respectively. Then $P_u, P_s$ being coprime, it is a standard result in linear algebra that $\ker P_u P_s (A)$ is the direct sum of $\ker P_u(A)$ and $\ker P_s(A)$. However by Cayley Hamilton theorem $\ker P_u P_s (A) = \bR^{d}$. Thus we obtain the following decomposition which is the very heart of hyperbolic dynamics.

\begin{Def}
	$\bR^{d}$ can be decomposed as the direct sum of two subspaces $E_s$ and $E_u$, invariant by $A$ such that the restriction of $A$ to each of these subspaces has eigenvalues of modulus $\abs{\lambda} < 1$ and $\abs{\lambda} > 1$ respectively. $E_s$, resp. $E_u$ is called the stable subspace, resp. unstable subspace. 
\end{Def}

\begin{lem}
	There exists a norm $\norm{\cdot}$ on $\bR^{d}$ such that for all $x \in \bR^{d}$ decomposing as $x = x_s + x_u$ with $x_s \in E_s$ and $x_u \in E_u$, $\norm{x} = \max(\norm{x_s}, \norm{x_u})$ and $\norm{A| _{E_s}} < 1$,  $\norm{A^{-1}|_{E_u}} < 1$, where the norms considered are the operator norms induced by $\norm{\cdot}$.
\end{lem} 

\begin{proof}
	Start from any norm $\norm{\cdot}$ on $\bR^{d}$. Since all the eigenvalues of $A| _{E_s}$ have modulus strictly smaller than $1$, there exists an integer $l \geq 1$ such that $\norm{A| _{E_s}^{l}} < 1$. This can be proved with the Jordan's normal form of the matrix. Similarly there exits $l' \geq 1$ such that $\norm{A| _{E_u}^{-l'}} < 1$. Taking the maximum of $l, l'$ we can suppose $l = l'$. Setting
	\[ 
	\norm{x}' := \max \left\{ \sum_{k=0}^{l-1} \norm{A| _{E_s}^{k} x_s}, \sum_{k=0}^{l-1} \norm{A| _{E_u}^{-k} x_u} \right\} 
	\]
	for all $x = x_s + x_u$ yields an adapted norm $\norm{\cdot}'$.
\end{proof}

A norm satisfying the properties of the lemma will be called adapted to $A$. From now on, we consider such an adapted norm $\norm{\cdot}$ and the associated quotient metric $d$. 

Finally let 
\begin{equation} 
	\lambda := \max \left(\norm{A| _{E_s}}, \norm{A^{-1}|_{E_u}} \right) < 1
\end{equation}

	\begin{prop}\label{prop:expansiveness} 
		Let $\epsilon > 0$ small enough so that $p$ is injective on $B\left(0, \epsilon \norm{A}\right)$. Consider $x, y \in \bT^{d}$ and define 
		\begin{equation*}
			K_{+} := \sup \{k \geq 0 \, | \,  \forall l \leq k,\ d(A^{l}x, A^{l} y) < \epsilon  \}.
		\end{equation*}
		where by convention the supremum is taken as $- \infty$ if the corresponding set is empty, that is if $d(x,y) \geq \epsilon$.	
		Otherwise let $v$ be the unique representative of $x-y$ in $B(0,\epsilon)$ and decompose it as $v = v_s + v_u$ with $v_s \in E_s$ and $v_u \in E_u$. Then 
		\begin{equation}
			\norm{v_u} < \lambda^{K_{+}}\epsilon \label{eq:expansiveness_future}
		\end{equation}
	In particular if $d(A^{k}x, A^{k} y) < \epsilon$ for all $k \geq 0$, then $y \in \Wseps(x)$.
	\end{prop}	
	
	
	The previous proposition is only concerned with future trajectories. Applying with $A^{-1}$ in place of $A$, it yields a dual statement on the past trajectories. Taken simultaneously with the forward statement, this yields a "two-sided" result for toral automorphisms. The last statement is what is traditionally refered to as the expansiveness property of toral automorphisms.	
	
	\begin{coroll}\label{coroll:expansiveness}
		Let $\epsilon > 0$ so that $p$ is injective on $B\left(0, \epsilon \max(\norm{A}, \norm{A^{-1}})\right)$. Consider 
		\begin{equation*}
			K := \sup \{k \geq 0 \, | \, \forall l \in [-k,k] \ d(A^{l}x, A^{l} y) < \epsilon  \}.
		\end{equation*}
		Then
		\begin{equation}\label{eq:expansiveness}
			d(x,y) < \epsilon \lambda^{K}.
		\end{equation}
		In particular if $d(A^{k}x, A^{k}y) < \epsilon$ for all $k \in \bZ$ then $x=y$. 
	\end{coroll}
	The smallest parameter $\epsilon$ satisfying Equation \eqref{eq:expansiveness} is called the expansiveness constant of $A$ and will be denoted $\ec$.
	
	\begin{proof}[Proof of Proposition \ref{prop:expansiveness}]
		Assume without loss of generality that $d(x,y) < \epsilon$ to have $K_{+} \geq 0$. Then for $k = 1$, $\norm{A^k v} < \epsilon'$ so $A^k v$ is the unique representative in $B(0, \epsilon')$ of $A^{k} (x-y) \in \bT^{d}$. Thus if $d(A^{k}x, A^{k}y) < \epsilon$ we deduce that $A^{k}v \in B(0, \epsilon)$. Iterating this argument, we get by induction that $A^{k}v \in B(0, \epsilon)$ for all $0 \leq k \leq K_{+}$. 
		
		Now since the norm is adapted to $A$, $\norm{v} = \max \left(\norm{v_s}, \norm{v_u} \right)$ and thus $A^{K_{+}} v_u$ belongs to $B(0, \epsilon)$. Then
		\[ 
		\norm{v_u} = \norm{A^{-K_{+}} A^{K_{+}} v_u} < \norm{A^{-K_{+}}| _{E_u}} \epsilon = \lambda^{K_{+}} \epsilon.
		\]
	\end{proof}

\subsection{Symbolic Representations}\label{subsec:symbolic}

The main ingredient in \cite{chung1987random} that allows to get a mixing time of $O( \log n \log \log n)$ is the use of the binary decomposition to study the orbits of elements of the form $\rho / n$ under the doubling map $x \mapsto 2x$. The theory of symbolic representations and Markov partitions provides a generalization of the binary decomposition to the multidimensional setting and beyond. By making the dynamical system $x \mapsto Ax$ essentially conjugate to a shift, this will permit the same kind of analysis as for the doubling map. 


The basic idea is as follows: given a finite partition $\cP =\{ P_1, \ldots, P_m \}$ of $\bT^{d}$, one may try to code an element $x \in \bT^{d}$ by the subsets of $\cP$ containing the orbit of $x$. For instance if for all $k \in \bZ$, $A^k x \in P_{x_k}$, then one may try to identify $x$ with the sequence $(x_k)_{k \in \bZ}$. 

There is a main obstacle to this: without restriction on $\cP$, two distinct elements may be coded by the same sequence. In the case of invertible hyperbolic automorphisms on the torus, it is possible to design well-suited families, called Markov partitions, to avoid this issue.

\bigskip

A Markov partition is not properly a partition: its sets, called rectangles, are closed sets with disjoint interiors but may have intersecting boundaries. Rectangles have additional properties but those are not necessary for the proofs of Theorems \ref{thm:upperbound} and \ref{thm:almostall}, so the definition is postponed to Section \ref{section:markov}. The diameter of a Markov partition is the maximal diameter of its rectangles. 

Given a Markov partition $\cR=\{ R_1, \ldots, R_m \}$, define the adjacency matrix $\cA = \cA(\cR)$ as the $m \times m$ matrix with entries
\begin{equation}\label{eq:adjacency_matrix}
\cA_{ij} = \left\{ \begin{array}{l l}
	1 & \text{if $\mathrm{int}(R_i) \cap A^{-1} \left(\mathrm{int}(R_j)\right) \neq \emptyset$} \\
	0 & \text{otherwise.}
\end{array} \right.
\end{equation}

Given an $m \times m$ adjacency matrix $\cA$, define 
\begin{equation}\label{eq:admissible_sequences}
\Omega_{\cA} := \left\{ \omega \in [m]^{\bZ}: \quad \forall k \in \bZ, \ \cA_{\omega_k, \omega_{k+1}} = 1  \right\}.
\end{equation}
Let $\theta : (\omega_k)_{k \in \bZ} \mapsto (\omega_{k+1})_{k \in \bZ}$ be the shift operator on $\Omega_{\cA}$. $[m]$ denotes the set of integers from $1$ to $m$ and $\Omega_{A}$ is given the product topology. 

The following results are \cite{bowen2008equilibrium}[Thm. 3.18, Thm. 3.12] applied to toral automorphisms, and will be proved in Section \ref{section:markov}. 

\begin{prop}\label{prop:existence_partitions}
	For any hyperbolic matrix $A \in \GL_{d}(\bZ)$, there exist Markov partitions of arbitrarily small diameter.
\end{prop}


\begin{prop}\label{prop:symbolic_twosided}
	Let $\cR$ be a Markov partition and $\cA = \cA(\cR)$ be the associated adjacency matrix. Provided the diameter of $\cR$ is small enough, for all $\omega \in \Omega_{\cA}$, the intersection $\bigcap_{k \in \bZ} A^{-k}(R_{\omega_k})$ is reduced to a single point, denoted $\pi(\omega)$. The map $\pi : \Omega_{\cA} \rightarrow \bT^{d}$ is continuous, finite-to-one, surjective and satisfies $\pi \circ \theta = A \circ \pi$.
\end{prop}

\begin{rk}
	In addition to being finite-to-one, the set of points which admit several coding sequences is actually negligible, in the sense that for all ergodic $\theta$-invariant measure with support $\Omega_{\cA}$, the set of points which have two images or more under $\pi$ has null measure; see \cite{benoist2002systemes}[Thm 12.4]. 
\end{rk}

A pair $(\Omega_{\cA},\pi)$, as given in Proposition \ref{prop:symbolic_twosided}, or simply the map $\pi$, will be called a symbolic representation of the dynamical system $x \mapsto Ax$. Similarly for $\xi \in \bT^{d}$, a sequence $\omega \in \pi^{-1}(\xi)$ will be called a symbolic representation of $\xi$. $\pi^{-1}(\xi)$ denotes here the inverse image of the singleton $\{ \xi \}$. We omit braces when considering inverse image of singletons.

We will finally need the two following results about Markov partitions which form Lemma 1 and Corollary 11 of \cite{bowen1970markovminimal}. They are proved in Section \ref{section:markov}.
	\begin{prop}\label{prop:symbolic_periodic}
		\begin{enumerate}[label=(\roman*)]
			\item If $x$ belongs to a rectangle $R \in \cR$, then there exists $\omega \in \Omega$ such that $x = \pi(\omega)$ and $\omega_{0} = R$.
			\item If $\pi(\omega)$ is a periodic point, then $\omega$ is periodic. 
		\end{enumerate}    
	\end{prop}

\subsection{Proof of Theorem \ref{thm:upperbound}}

Let $\rho \in \bT_n$ and $t \geq 1$. From the formula of the Fourier transform, multiply Equation \eqref{eq:fourier_general} by its conjugate to obtain
\begin{align} \label{eq:fourier_square}
	\abs{\widehat{P^{t}}(\rho)}^{2} &= \prod_{j=0}^{t-1} \abs{\hmu( (A^{\top})^j \rho)}^2 \\
	&= \prod_{j=0}^{t-1} \left( \sum_{x,y \in \supp \mu} \mu(x) \mu(y) e^{2 i \pi \innerprod{A^j (x-y)}{\rho} /n} \right). 
\end{align} 

For $k \geq 0$ and $\xi \in \bT^{d}$, define
\begin{equation*}
	f_k(\xi) := \sum_{x,y \in \supp \mu} \mu(x) \mu(y) e^{2 i \pi \innerprod{A^k (x-y)}{\xi}} 
\end{equation*}
so that $\abs{\widehat{P^{t}}(\rho)}^{2} = \prod_{k=0}^{t-1}  f_k(\rho/n)$.
For all $k, l \geq 0$, $0 \leq f_l \leq 1$ and $f_k \circ (A^{\top})^{l} = f_{k+l}$ hence for $k \in [0,d] = \{0, 1, \ldots, d \}$,
\begin{align*}
	\abs{\widehat{P^{t+d}}(\rho)}^{2} &= \prod_{l=0}^{k-1} f_{l}(\rho / n) \prod_{j=k}^{t+d-1} f_{j}(\rho / n) \\
	&\leq \prod_{j=0}^{t+d-k-1} f_{j+k}( \rho / n) \\
	&\leq \prod_{j=0}^{t+d-k-1} f_{k}( (A^{\top})^{j} \rho / n) \\
	&\leq \prod_{j=0}^{t-1} f_k \left( (A^{\top})^{j} \rho / n \right) .
\end{align*}
In particular, if we set 
\begin{equation}\label{eq: def_f}
	f := \min_{k \in [d]} f_k
\end{equation}
then we have for all $t \geq 0$,
\begin{equation}\label{eq:majoration_fourier}
	\abs{\widehat{P^{t+d}}(\rho)}^{2} \leq \prod_{j=0}^{t-1} f \left( (A^{\top})^{j} \rho / n \right). 
\end{equation}
Furthermore, the $f_k$ are continuous and thus so is $f$.

Consider now
\begin{equation}\label{eq:subgroupH}
	H := \langle \bigcup_{k \geq 0} A^{k} (\supp \mu - \supp \mu) \rangle
\end{equation}
the smallest subroup of $\bZ^{d}$ that contains $\supp \mu - \supp \mu$ and is invariant by $A$. By Cayley-Hamilton's theorem, we can in fact restrict the powers of $A$ to have exponent less than $d$. 

Let $u_1, \ldots, u_d$ be generators of $H$, although at this stage one does not require $H$ to have rank $d$. They can be chosen of the form $u_i = A^{k_i} (x_i - y_i)$ with $k_i \leq d$ and $x_i, y_i \in \supp \mu$.  Then let 
\[ 
	\alpha := \min \bigcup_{i} \left\{ \mu(x_i), \mu(y_i) \right\} > 0
\] 
and
\[ 
	W = \{ \xi \in \bT^d: \ \forall h \in H, \  \innerprod{h}{\xi} \in \bZ  \}.
\]

$W$ is a closed set of $\bT^{d}$ and is invariant by $A^{\top}$ since $H$ is invariant by $A$. If $\xi \notin W$, then one can find $h \in H$ such that $\innerprod{h}{\xi} \notin \bZ$. However $h$ being a linear combination with integer coefficients of the $A^{k_i}(x_i-y_i)$, we can directly assume that $h = A^{k_i}(x_i-y_i)$.
Then we deduce that 
\begin{align}
 f(\xi) &\leq 1 - 2 \mu(x_i) \mu(y_i) + 2 \mu(x_i) \mu(y_i) \cos 2 \pi \innerprod{A^{k_{i}} (x_i-y_i)}{\xi}  \nonumber \\
 &\leq 1 - 2 \alpha^{2} \abs{ 1- \cos 2 \pi \innerprod{A^{k_i} (x_i-y_i)}{\xi} \label{eq:majoration_exp} } \\
 & < 1 \nonumber.
\end{align}
By continuity of $f$ we get for all $\eta > 0$
\begin{equation}
	\sup_{\xi: \ d(\xi, W) \geq \eta}	f(\xi) \leq \gamma
\end{equation}
for some $\gamma= \gamma(\eta) < 1$.
Thus our goal is now to prove that the points on the orbit of $\rho/n$ spend a considerable amount of time away from the “bad set” $W$.

For this, we will use the symbolic representation of the orbits with an explicit description of the set $W$. So far Fourier transforms were expressed in the canonical basis but this choice is not well-suited for expressing the elements of $H$. On the other hand by considering a basis adapted to $H$, $H$ becomes generated by multiple of the basis vectors the set $W$ takes a very simple form. We use the following well known result (see \cite[Chap. III Thm. 7.8]{lang2002algebra}) which can be deduced from the Smith normal form of a matrix. We recall a basis of a subgroup $H \subseteq \bZ^{d}$ is a minimal generating family of $H$. The terminology comes from the fact that subgroups of $\bZ^{d}$ are free modules over the ring $\bZ$.

\begin{lem}\label{lem:abelian_subgroups}
	Let $H$ be a subgroup of $\bZ^{d}$. Then there exists a basis $(u_i)_{i=1}^{d}$ of $\bZ^{d}$ and positive integers $a_1, \ldots, a_k$, such that $a_i$ divides $a_{i+1}$ for all $i=1, \ldots, k-1$ and $a_1 u_1, \ldots, a_k u_k$, forms a basis of $H$. The family $(a_i)_{i=1}^{k}$ is uniquely determined by the previous conditions.
\end{lem}

\paragraph{Change of basis}

By the previous lemma, there exists a basis $(u_i)_{i \in [d]}$ of $\bZ^{d}$ and a family of positive integers $(a_i)_{i \in [d]}$  such that the subgroup $H$ has basis $(a_i u_i)_{i=1}^{d}$. Let $Q$ be the matrix formed from the vectors $u_i$, expressed in the canonical basis, so that $Q e_i = u_i$ for all $i \in [d]$. 

We can now express everything in the basis $(u_i)_{i \in [d]}$: this comes down to replace the map $A$ by $Q^{-1} A Q$ which has the same desired properties as $A$: it is in $\GL_{d}(\bZ)$ and hyperbolic, while representation $\rho$ has to be pulled-back by $Q$, ie it has to be replaced by $\rho \circ Q$. Notice now that the upper bound lemma \eqref{eq:upperboundlemma} involves summing over all non-trivial representations, which form a set left invariant by the pull-back operation, so we can simply forget about this pull-back.

Hence, we can now suppose that all vectors expressions are given in a basis adapted to $H$, so $H$ is generated by $(a_i e_i)_{i \in [d]}$. 
Then one easily obtains from the definition of $W$ that

\begin{equation}\label{eq:setW}
	W = \left\{ \left( \frac{k_i}{a_i} \right)_{i=1}^{d}, k_i \in [0,a_i - 1] \ \forall i \in [d] \right\}.
\end{equation}

\begin{lem}\label{lem:unstable_lower_bound}
	There exists constants $c_1, c_2 > 0$ such that the following holds. Let $\rho \in \bT_n \smallsetminus \{0\}$ and $y$ be either a point of $W$ or $y = \rho' / n$ with $\rho' \in \bT_n \smallsetminus \{0\}$ distinct from $\rho$. Then either $d(\rho / n, y) \geq \epsilon$ or, letting $v$ be the unique representative of $\rho / n - y$ in $B(0, \epsilon)$,
	\begin{equation}\label{eq:unstable_lower_bound}
		\norm{v_u} \geq \frac{c_1}{n^{c_2}}.
	\end{equation} 
	where $v=v_s + v_u$ is the decomposition of $v$ into stable and unstable component.
\end{lem}

\begin{proof}	
	Consider the case of $y=w\in W$. Then there exist integers $(k_i)_{i=1}^{d}$ such that $w= \left( \frac{k_i}{a_i} \right)_{i=1}^{d}$ while $\rho = (\rho_i)_{i \in [d]} \in \bT_{n} \smallsetminus \{0\}$. Suppose that $d(\rho/n, w) < \epsilon$. Assuming that $n$ is coprime with the $a_i$, it is impossible to have $\rho / n = w$ and furthermore any non-zero coordinate has absolute value
	\[ 
		\abs{(\rho / n - w)_i} = \abs{\frac{\rho_i a_i - k_i n}{a_i n}} \geq \frac{1}{a_i n}.
	\]
	By equivalence of norms we deduce there exists a constant $c_1=c_1(A, a_i) \in (0,1)$ such that 
	\[ 
		\norm{v} = \norm{\rho/n - w} \geq \frac{c_1}{n}.
	\]
	Since the norm is adapted, this implies that $\norm{v_u} \geq c_1 / n$, in which case the claimed result holds, or $\norm{v_s} \geq c_1 / n$. In that case, one still has $\norm{v_s} \leq \epsilon$. Thus for $k > \log n / \log (\lambda^{-1})$ and $n$ large enough, 
	\[ 
		\norm{A^{k} v_s} \leq \lambda^{k} \norm{v_s} < \frac{c_1}{n}.
	\]
	Now if $\norm{A^{l} v} < \epsilon$ for all $0 \leq l \leq k$, then $A^{k} v$ is the representative in $B(0, \epsilon)$ of $A^{k} \rho/n - A^{k} w$. Since the sets $\{ \rho'/n, \rho' \in \bT_n \smallsetminus \{0\} \}$ and $W$ are left invariant by $A$, $A^{k} v$ must also satisfy: 
	\[ 
		\norm{A^{k} v} \geq \frac{c_1}{n}.
	\]
	From the choice of $k$ we deduce that $\norm{A^{k} v_u} \geq c_1 / n$. Then $\norm{A^{k} v_u} \leq \norm{A}^{k} \norm{v_u}$ yields
	\[ 
		\norm{v_u} \geq \frac{c_1}{\norm{A}^{k} n} = \frac{c_1}{n^{1+\log \norm{A} / \log \lambda^{-1}}}.
	\]
	Finally if $A^{l} v$ leaves the ball $B(0, \epsilon)$ for some $0 \leq l \leq k$, it can only expand in the unstable direction so necessarily $\norm{A^{l} v_u} \geq \epsilon \geq c_2 / n$ for $n$ large enough and the same conclusion holds.   
	
	The case $y =\rho'/n$ is proved similarly.
\end{proof}

Let us now turn to symbolic representations. Consider a Markov partition $\cR$ associated to $A^{\top}$, of diameter $\delta$ determined later on, as given by Proposition \ref{prop:existence_partitions}. Let $m := \abs{\cR}$ be the number of rectangles and $(\Omega, \pi)$ the symbolic representation given by the Markov partition (Proposition \ref{prop:symbolic_twosided}). 
Consider the set $\cR_0$ of rectangles which contain a point of $W$ and set $\cR_1 := \cR \smallsetminus \cR_0$. We define
\[ 
	\delta_0 := \min \left(\ec, \frac{\min_{x \neq y \in W} d(x,y)}{1+\norm{\AT}} \right),
\]
which is positive by finiteness of $W$.
From now on, we identify $[m]$ with the set of rectangles $\cR$. In particular, given a sequence $\omega \in \Omega$, we will write $\omega_k \in \cR_{0}$ to indicate that $R_{\omega_k} \in \cR_{0}$. Similarly, we may use the words letters and rectangles interchangeably.

\begin{lem}\label{lem:symbolic_W}
	Suppose the diameter of the partition satisfies $\delta < \delta_0 $. Then
	\begin{enumerate}[label=(\roman*)]
		\item for all $i \in \cR_0$, there exists a unique $j \in \cR_0$ such that $\cA(i,j) = 1$;
		\item \[ 
		W = \left\{ \pi(\omega): \forall k \in \bZ,  \omega_{k} \in \cR_0\, \right\}.
		\]
	\end{enumerate}
\end{lem}

\begin{proof}
	\begin{enumerate}
		We start proving (ii). 
		
		One inclusion in (ii) is easy. For all $w \in W$ by surjectivity of $\pi$ there exists $\tau \in \Omega$ such that $w = \pi(\tau)$. Then by definition of $\pi$, $(A^{\top})^{k} w \in R_{\tau_k}$ for all $k \in \bZ$. However $W$ is invariant by $A^{\top}$, so by definition of $\cR_0$ $\tau_k \in \cR_0$. This holds for every sequence $\tau \in \pi^{-1}(w)$.
		
		Conversely, suppose $\tau \in \Omega$ is such that $\tau_k \in \cR_0$ for all $k \geq 0$ and let $x = \pi(\tau)$. Then for all $k \geq 0$ there exists $w_k \in R_{\tau_k} \cap W$ at distance at most $\delta$ from $(\AT)^{k} x$. The sequence $(w_{k})_{k \geq 0}$  is a $\delta(1+\norm{\AT})$-pseudo orbit, in the sense
		\begin{align*}
			\forall k \geq 0 \quad d(w_{k+1}, \AT w_k) &\leq d(w_{k+1}, (\AT)^{k+1} x) + d(\AT (\AT)^{k} x, \AT w_k) \\
			&\leq (1+\norm{\AT}) \delta.
		\end{align*}
		Therefore if $\delta (1+\norm{\AT}) < \min_{x \neq y \in W} d(x,y)$, the sequence $(w_k)_{k \in \bZ}$ is necessary the orbit of a point $w \in W$. Then we obtain $d((\AT)^k x, (\AT)^k w) \leq \delta$ for all $k \in \bZ$. If $\delta \leq \ec$, Proposition \ref{prop:expansiveness} implies $x =w$. Hence $x \in W$.
		
		We now prove (i). Consider $i \in \cR_0$. By definition $R_i$ must contain a point $w \in W$. From Proposition \ref{prop:symbolic_periodic} (i), $w$ admits a symbolic representation $\tau$ such that $\tau_0 = i$. Let $j := \tau_1$. Then $\cA(i,j) = 1$ and $R_j$ contains $A^{\top}w \in W$ so $j \in \cR_0$. This proves existence.
		
		Suppose now $\cA_{i,j'} = 1$ with $j' \in \cR_0$. By the existence results one can find for all $l \geq 2$ $\tau^{(l)} \in \Omega$ such that $\tau^{(l)}_0 = i, \tau^{(l)}_1 = j'$ and $\tau^{(l)}_k \in \cR_0$ for all $k \leq l$. By compactness of $\Omega$, $(\tau^{(l)})_{l}$ admits a subsequence converging towards $\tau' \in \Omega$, which satisfies $\tau'_0 =i, \tau'_1 = j'$ and $\tau'_k \in \cR_0$ for all $k \geq 0$. By (ii), $\pi(\tau') =: w'$ is a point of $W$ which is also in the rectangle $R_i$. However notice that $\delta < \delta_0 < \min_{x \neq y \in W} d(x,y)$ so every rectangle of the Markov partition contains at most one point of $W$ and necessarily $w' = w$. We thus proved that $\tau, \tau'$, or any other sequence in $\cR_{0}^{\bZ}$ starting with $i$, represent the same element $w$. A last observation to make is that by Proposition \ref{prop:symbolic_periodic} (ii), $\tau, \tau'$ are periodic sequences. Thus $\tau, \tau'$ are the concatenation infinitely many times of a same block of finite size, say $B$ for $\tau$ and $B'$ for $\tau'$. Substituting any block $B$ by a block $B'$ and vice versa, one obtains other representations of $w$. Since $\pi$ is finite-to-one, there can only be finitely many such representations, which implies $B=B'$ and in particular $j = j'$. Hence uniqueness. 
	\end{enumerate}

\end{proof}

Given an integer $k \geq 1$, $\xi \in \bT^{d}$ and a symbolic representation $\omega=(\omega_i)_{i \geq 0} \in \pi^{-1}(\xi)$ of $\xi$. we decompose $\omega$ into blocks of $k$ letters: $B_{i}(\omega) := \omega_{(i-1)k} \ \cdots \ \omega_{ik - 1}$, which we call $k$-blocks. We set $B_i(\xi) := \{ B_{i}(\omega), \omega \in \pi^{-1}(\xi) \}$ to be the set of all blocks that are the $i$-th block of some symbolic representation of $\xi$. When considering $\xi = \rho / n$, we may simply write $B_{i}(\rho)$ instead of $B_{i}(\rho / n)$.

\begin{lem}\label{lem:blocks}
	Let $\cR$ be a Markov partition of diameter $\delta < \delta_{0}$, and some constant $c$. For all $n\geq 1$ let
	\begin{equation}\label{eq:size_kblocks}
		k = 1+\lceil \frac{c_2}{\log \lambda^{-1}} \log n \rceil.
	\end{equation}
	and decompose symbolic representations into blocks of $k$ letters. For large enough $n$:
	\begin{enumerate}[label=(\roman*)]
		\item For all $\rho \in \bT_n \smallsetminus \{ 0 \}$, every block of $B_1(\rho)$ contains at least one letter in $\cR_1$.
		\item The sets of blocks $(B_{1}(\rho)_{\rho \in \bT_n})$ are all disjoint.
		\item The family $(B_{i}(\rho)_{\rho \in \bT_n})$ is independent of $i \in \bZ$.
	\end{enumerate}
\end{lem}

\begin{proof}
	\begin{enumerate}[label=(\roman*)]
		\item Consider any symbolic representation $\omega$ of $\rho \neq 0$. If $\omega_0, \ldots, \omega_{k} \in \cR_{0}$, from the proof of Lemma \ref{lem:symbolic_W} (i), there exists $\omega' \in \Omega$ with $\omega'_l = \omega_l$ for all $0 \leq l \leq k$ and $\omega'_l \in \cR_0$ for all $l \in \bZ$. Lemma \ref{lem:symbolic_W} (ii) tells us that $\pi(\omega') =: w$ is a point of $W$. $w$ is at distance at most $\delta$ from $\rho / n$ so we can consider the unique representative $v$ of $\rho/n - w$ in $B(0, \ec)$. Since $\omega$, $\omega'$ coincide up to rank $k$, \eqref{eq:expansiveness_future} yields $\norm{v_u} \leq \ec \lambda^{k}$, which contradicts \eqref{eq:unstable_lower_bound} for $k \geq c_2 / \log \lambda^{-1} \log n$ and $n$ large enough. Hence for $k$ as in \eqref{eq:size_kblocks} $\omega_l \in \cR_1$ for some $0 \leq l \leq k-1$.
		
		\item The proof is the same as (i), since if $\rho, \rho' \in \bT_n$ are distinct, $\rho/n - \rho'/n$ has unstable component lower bounded by $\min(\ec, c_1 n^{-1-c_2})$ by Lemma \ref{lem:unstable_lower_bound}.
		
		\item The matrix $A^{\top}$ is a bijection on $\bT_{n} \smallsetminus \{0 \}$, so the shift by $k$ letters on $\Omega$ induces a permutation on the blocks.    
	\end{enumerate}
\end{proof}

The rest of the proof is now similar to the original argument of Chung, Diaconis and Graham \cite{chung1987random}. Let $k$ be as in \eqref{eq:size_kblocks}.

By definition, rectangles of $\cR_1$ are closed sets which contain no point of $W$, and $\cR_1$, $W$ are both finite sets, thus $\eta := \min_{R \in \cR_1} d(R,W)$ must be positive. By equation \eqref{eq:majoration_exp} and the continuity of the function $f$ \eqref{eq: def_f}, there exists $\gamma= \gamma(\eta, \alpha) \in (0,1)$ such that $f$ is bounded by $\gamma$ on the compact set $\{\xi \in \bT_{d}, d(x,W) \geq \eta \}$. In particular $f(\rho / n) \leq \gamma$ if $\rho / n$ is contained in a rectangle of $\cR_1$.

Combining this observation with Lemma \ref{lem:blocks} and equation \eqref{eq:majoration_fourier}, we deduce that for all $\rho \in \bT_{n} \smallsetminus \{ 0\}$, for all $r \geq 0$
\begin{align*}
	\abs{\widehat{P^{rk'+d}}(\rho)}^{2} &\leq \gamma^{g(\rho / n)} \\
	&\leq \prod_{i=1}^{r} \gamma^{g(B_{i}(\rho/n))}
\end{align*}
where $g(\rho / n)$ is the number of letters in $\cR_1$ appearing in the first $rk$ letters of any symbolic representation of $\rho / n$, and $g(B_{i}(\rho/n))$ is the maximal number of letters in $\cR_1$ appearing in the $i$-th $k'$-block of some symbolic representation of $\rho/n$. 

As in \cite{chung1987random}, we then make use of the following interchange lemma in order to regroup similar blocks when summing over $\rho \neq 0$: for all $a \leq a'$, $b \leq b'$,
\begin{equation}\label{eq:interchange}
 	\gamma^{a+b'} + \gamma^{a'+b} \leq \gamma^{a+b} + \gamma^{a'+b'}
\end{equation}
Using point (iii) of Lemma \ref{lem:blocks}, we deduce the bound
\begin{align*}
	\sum_{\rho \neq 0} \abs{\widehat{P^{rk+d}}(\rho)}^{2} \leq \sum_{\rho \neq 0}  \gamma^{ r g(B_{1}(\rho/n))}
\end{align*}
Then, using point (i) and (ii) of Lemma \ref{lem:blocks}, this sum can be bounded by the sum over all blocks of length $k$ with at least one letter in $\cR_1$. Let $m_0 := \abs{\cR_0}$, $m_1 = \abs{\cR_1} = m - m_0$. Once the positions of the rectangles in $\cR_1$ have been determined, there are $m_{1}^{j}$ choices of such rectangles. On the other hand, Lemma \ref{lem:symbolic_W} (i) shows that the sequences containing only rectangles in $\cR_{0}$ are completely determined by their first letter. Thus the total number of blocks with $j$ letters in $\cR_1$ is upper bounded by $\binom{k}{j} (m_{1} m_{0})^{j}$. Hence
\begin{align*}
	\sum_{\rho \neq 0} \abs{\widehat{P^{rk+d}}(\rho)}^{2} &\leq \sum_{\rho \neq 0} \gamma^{ r g(B_{1}(\rho/n))} \\
	&\leq \sum_{j=1}^{k} \binom{k}{j} (m_{1} m_{0})^{j} \gamma^{ r j} \\
	&= \left( \left(1 + m_{0} m_{1} \gamma^{r} \right)^{k} - 1 \right) \\
	&\leq \left( e^{k m_{0} m_{1} \gamma^{r}} - 1 \right) .
\end{align*}

Finally use the upper bound lemma \eqref{eq:upperboundlemma} to get
\begin{equation}
	\TV{P^{rk + d}(0, \cdot) - U} \leq \frac{1}{4} \left( e^{m_0 m_1 k \gamma^{r}} - 1 \right).
\end{equation}
Thus for all $s > 0$ and $r \geq  \frac{\log (m_0 m_1 k)}{\log (\gamma^{-1})} +s$ we obtain 
\[ 
	\TV{P^{rk' + d}(0, \cdot) - U} \leq \frac{1}{4} \left( e^{\gamma^{s}} - 1 \right).
\] 
Since $k = O(\log n)$, this yields the result.

\section{Necessary condition for irreducibility and aperiodicity}\label{section:irr_aper}

In this section we prove that the condition on the subgroup $H$ given in Theorems \ref{thm:upperbound} and \ref{thm:almostall} is necessary to ensure irreducibility and aperiodicity of the random walk. Proposition \ref{prop:cns_aper} and its proof should be reminiscent of the following result for convolution random walks on groups.  

For a random walk on a finite group $\Gamma$ defined as the product of iid random increments of law $\mu$, irreducibility is equivalent to the fact that the subgroup generated by $\supp \mu$ is the whole group $\Gamma$. Furthermore if it is irreducible, it is aperiodic if and only if $\supp \mu$ is not contained in the coset of a proper normal subgroup, which is also equivalent to (see for example \cite[p.97]{mukherjea2006measures})
\[
\bigcup_{k \geq 0} (\supp \mu)^{-k} (\supp \mu)^{k} = \Gamma,
\]
where for all subset $K \subseteq \Gamma$ and $l \geq 0$, $K^{l} := \{ x_1 \cdots x_k,  \forall i \ x_i \in K \}$, $K^{-1} := \{x^{-1}, x \in K \}$ and $K^{-l} := (K^{-1})^{l}$. In what follows, additive notation is used for group operation.

\begin{prop}\label{prop:cns_aper}
	Let $H$ be the smallest $A$-invariant subgroup of $\bZ^{d}$ that contains $\supp \mu - \supp \mu$. Suppose it is generated by $(a_i u _i)_{i=1}^{l}$ for some basis $(u_i)_{i=1}^{d}$ of $\bZ^{d}$ and integers $(a_i)_{i=1}^{l}$.  If the random walk $X_{t}$ is irreducible and aperiodic then necessarily $l = d$ and $n$ is coprime with all the $a_i$.
\end{prop}

\begin{proof}

	The subgroup $H$ is described by \eqref{eq:subgroupH}. Let $P^{t}$ denote the law of $X_t$ at time $t$. Consider the subgroup $N := H \mod n$ of $\bT_{n}$ defined as the image of $H$ under the natural projection. We claim the following: 
	\begin{claim}
		\[ 
		N = \bigcup_{k \geq 0} \supp P^{k} - \supp P^{k}  
		\] 
	\end{claim}
	\begin{claim}
		If $X_t$ is irreducible and aperiodic then $N = \bT_n$.
	\end{claim} 
	From the second claim and Lemma \ref{lem:abelian_subgroups} we easily deduce the result. Consider a basis $(u_i)_{i=1}^{d}$ of $\bZ^{d}$ and integers $a_1, \ldots, a_l$ such that $(a_i u_i)_{i=1}^{l}$ forms a basis of $H$. Then $N$ consists of the projections on $\bT_n$ of all linear combinations of the $a_i u_i$ with coefficients in $[0,n-1]$. In particular $N$ has at most $n^{l}$ points, which proves the necessity of $l = d$. 
	
	Suppose now $l = d$. The basis $(u_i)_{i=1}^{d}$ actually yields an isomorphism between $N$ and the subgroup of $\bT_n$ generated by elements $(0, \ldots, 0, a_i, 0, \ldots, 0)$ where $a_i$ is in position $i$. As $a_i$ generates $\bZ /  n \bZ$ if and only if $n$ is coprime with $a_i$, we deduce $N =\bT_n$ if and only if $n$ is coprime with all the $a_i$, which is the desired result.  
	
	Let us now prove the claims. For all $k \geq 0$, let $\mu_{k} := A^{k} \mu$. $N$ is by definition the subgroup generated by $\bigcup_{k \geq 0} \supp \mu_k - \supp \mu_k$.
	For $k \geq 1$, the random walk can be written
	\[ 
	X_{t} = \sum_{i=1}^{k} A^{k-i} g_i \mod n
	\]
	with $g_{i} \in \supp \mu$ for all $i=1, \ldots, k$, so $P^{k} = \mu_0 \ast \cdots \ast \mu_{k-1}$, from which we deduce $\supp \ P^{k} - \supp \ P^{k} \subseteq N$ for all $k \geq 0$.
	
	Conversely, notice that $0 \in \supp P^k - \supp P^k$ for all $k \geq 0$. Thus, given $x_k - y_k \in \supp \mu_k - \supp \mu_k$, we can choose $x_i = y_i \in \supp \ \mu_{i}$ for $i = 0, \ldots, k-1$ and write $x_k - y_k = x_k -y_k + \sum_{i \leq k-1} (x_i - y_i) $ to obtain that $\supp \mu_k - \supp \mu_k \subseteq \supp P_k - \supp P_k$. Since $N$ is generated by the $\supp \mu_k - \supp \mu_k$, it suffices to prove that the set $\bigcup_{k \geq 0} \supp \ P^{k} - \supp \ P^k$ is a group.
	
	Stability under inverse is clear, whereas stability under addition will be derived from of the following observation: since $\det A = \pm 1$, $A$ induces an automorphism of the finite group $\bT_{n}$, a power of which is the identity. Consequently, the sequence of image measures $(\mu_k)_{k \geq 0}$ is periodic, say of period $m$. 
	Now for $l \geq 0$, using the trick that $0 \in \supp \ \mu_l - \supp \mu_l$ for all $l \geq 0$, we see that $\supp \ P^k - \supp \ P^{k} \subseteq \supp \ P^{l} - \supp \ P^{l}$ for all $l \geq k$. We can consider in particular $l = - 1 \mod m$, so that for all $k' \geq 0$, $P^{l+k'} = P^{l} \ast \mu^{0} \ast \cdots \ast \mu^{k'-1} = P^{l} \ast P^{k'}$. All in all, this shows that $\supp \ P^{k} - \supp \ P^{k} + \supp \ P^{k'} - \supp \ P^{k'} \subseteq \supp \: P^{l+k'} - \supp P^{l+k'}$, which is the desired addition property.
	
	Finally, the second claim is easily deduced from the first.
	Consider any sequence $(x_t)_{t \geq 0}$ such that $x_t \in \supp P^{t}$ for all $t \geq 0$. Then by the characterization of $N$, for all $t \geq 0$, $X_t$ is included in the coset $x_t + N$. Hence if $N \neq \bT_{n}$ the random walk cannot be irreducible and aperiodic. 
	
\end{proof}

\section{An upper Bound for almost all $n$: proof of Theorem \ref{thm:almostall}}\label{section:almostall}

In this section we reuse the results of Section \ref{section:upper}. The proof follows the same argumentation as in \cite{chung1987random, hildebrand1996random}.  For all $n \geq 1$, let $U_n$ be the uniform measure on $\bT_n$ and $P_{n}^{t}$ the law at time $t$ of the random walk on $\bT_n$ defined by $\eqref{eq:def_RW}$. 

Fix an integer $k \geq 2$, $t = \Omega(k)$ and $a_i$ be defined as in Lemma \ref{lem:unstable_lower_bound}. We will consider at once all integers $n \in N_k := \{ c \tlambda^{k-1} \leq m < c \tlambda^{k} \text{, $m$ coprime with all $a_i$} \}$, where $c > 0$ is some constant which can be explicited and $\tlambda = \lambda^{-1/c_2}$. The choice of $N_k$ (and thus of $c)$ is made so that 
\[ 
	\ec \lambda^{k} < \frac{c_1}{n^{c_2}} \leq \ec \lambda^{k-1}
\] 
whenever $n \in N_k$, $c_1, c_2$ being the constants appearing in Lemma \ref{lem:unstable_lower_bound}.
In the sequel, $n$ is implicitely taken in this set.  By the upper bound Lemma \eqref{eq:upperboundlemma} and \eqref{eq:majoration_fourier}, we can upper bound $\sum_{n} \TV{P_{n}^{t+d} - U_n}$ by
\[ 
	S := \sum_{n} \sum_{\rho \in \bT_{n} \smallsetminus \{ 0 \}} \prod_{j=0}^{t-1} f((A^{\top})^{j} \rho / n))
\]
Since $n$ is not fixed anymore, it may happen that $\rho / n = \rho' / n'$ for distinct pairs $(\rho,n), (\rho',n')$. We regroup these terms as follows. 

Given $r \in \bN^{d}$ a vector with integer entries and an integer $s > 0$, define $\gcd(r,s) $ as $\gcd(r_{1}, \ldots, r_d, s)$. Any vector $\rho / n$ can be rewritten uniquely under the form $r / s$ with $\gcd(r,s) = 1$. Furthermore $s$ is necessarily coprime with the $a_i$ if $n$ is. Given $r$ and $s$ such that $\gcd(r,s) = 1$, let $M(r/s)$ be the number of pairs $(\rho, n)$ such that $\rho / n = r /s$. Such $\rho / n$ are obtained simply by multiplying numerators and denominators of $r/s$ by an integer factor necessarily smaller than $c \tlambda^{k} / s$, hence $M(r/s) \leq c \tlambda^{k} / s$. Then rewrite the sum as 
\begin{align*}
	S &= \sum_{1 \leq s < c \tlambda^{k}} \sum_{\substack{r \in \bT_{s} \smallsetminus \{ 0 \} \\ \gcd(r,s) = 1}} M(r/s) \prod_{j=0}^{t-1} f((A^{\top})^{j} r / s) \\
	&= \sum_{l = 1}^{k} \sum_{c \tlambda^{l-1} \leq s < c \tlambda^{l}} \sum_{\substack{r \in \bT_{s} \smallsetminus \{ 0 \} \\ \gcd(r,s) = 1}} M(r/s) \prod_{j=0}^{t-1} f((A^{\top})^{j} r / s) \\
	&\leq c \tlambda^{k} \sum_{l = 1}^{k} \sum_{c \tlambda^{l-1} \leq s < c \tlambda^{l}} \sum_{\substack{r \in \bT_{s} \smallsetminus \{ 0 \} \\ \gcd(r,s) = 1}} \frac{1}{s} \prod_{j=0}^{t-1} f((A^{\top})^{j} r / s).
\end{align*}
As we did for $n$, the integers $s$ considered will now be implicitely assumed to be coprime with the $a_i$.

For every pair $(r,s)$ the product can be bounded by $\gamma^{g(r/s)}$, where we recall $g(r/s)$ is the maximal number of rectangles in $\cR_1$ appearing in the $t$ first letters of some symbolic representation of $r/s$. What follows now is the same application of Lemma \ref{lem:blocks} as we did in Section \ref{section:upper} but with $l$ and $s$ in place of respectively $k$ and $n$.
Define for all $l \geq 1$, 
\[ 
	Q_l := \left\{ (r,s): c \tlambda^{l-1} \leq s < c \tlambda^{l}, r \in \bT_s \smallsetminus \{ 0\}, \gcd(r,s) = \gcd(s,a_i) = 1 \right\}.
\]
Then
\[ 
	S \leq c \tlambda^{k} \sum_{l=1}^{k} \sum_{(r,s) \in Q_l} \frac{1}{s} \gamma^{g(r/s)}.
\]

Decompose the symbolic representations of $r/s$ into blocks of $l_{1} = C l$ letters and set $t' = \lfloor t / l_{1} \rfloor$. Since $s$ is coprime with the $a_i$, we can apply Lemma \ref{lem:blocks} with $s$ instead of $n$ basically in the same way. The only difference concerns point (ii), as the distance between distincts points is now lower bounded by
\[ 
d(r/s, r'/s') \geq \frac{c'}{s s'}
\]
for some constant $c'$. Arguing as in the proof of Lemma \ref{lem:unstable_lower_bound}, this implies a lower bound on the unstable component $v_u$ of $r/s - r'/s'$, namely $\norm{v_u} \geq c' / (ss')^{c_2}$ for some possibly different constant $c'$. The definition of $Q_l$ is made so that $(r,s) \in Q_l$ implies $\ec \lambda^{l} < c_1 / s^{c_2}$. Thus choosing $C$ large enough in the definition of blocks, one has $\ec \lambda^{l_1} < c_1 / (ss')^{c_2}$ so by Proposition \ref{prop:expansiveness} symbolic representations of $r/s, r'/s'$ must have distinct blocks.

Then use the interchange inequality \eqref{eq:interchange} to obtain
\begin{align*}
	S \leq c \tlambda^{k} \sum_{l=1}^{k} \sum_{(r,s) \in Q_l} \frac{1}{s} \gamma^{t' g(B_{1}(r/s))} .
\end{align*}
For $(r,s) \in Q_l, \frac{1}{s} \leq c^{-1} \tlambda^{-l+1} = O(\tlambda^{-l})$. On the other hand, we can upper bound the sum over $Q_l$ by the sum over all blocks of length $l_{1}$, thus
\begin{align*}
	\sum_{(r,s) \in Q_l} \frac{1}{s} \gamma^{t' g(B_{1}(r/s))} &\leq \tlambda^{-l} \sum_{j=1}^{l_{1}} \binom{l_{1}}{j} \gamma^{t' j} (m_0 m_1)^{j} \\
	&= \tlambda^{-l} \left( (1 + m_0 m_1 \gamma^{t'})^{l_{1}} - 1 \right).
\end{align*}

We now consider two different regimes for $l$.

If $l \leq k / \log k$, then $t' \rightarrow \infty $ as $k \rightarrow \infty$. Now for large $t'$ we can use Taylor's theorem to bound $(1 + m_0 m_1 \gamma^{t'})^{l_{1}} - 1$ by $O(l_{1} \gamma^{t'})$. As $l_1 = C \log l$ this yields
\begin{align*}
	\sum_{l \leq k / \log k} \tlambda^{-l} \left( (1 + m_0 m_1 \gamma^{t'})^{l_{1}} - 1 \right) &\leq O \left( \sum_{l \leq k / \log k} \tlambda^{-l} l \ \gamma^{t/(Cl)} \right) \\
	&\leq O \left( \sum_{l \geq 0} \tlambda^{-l} l \  \gamma^{t \log k/(C k)} \right)\\
	&\leq O(\gamma^{t \log k / (C k)} )
\end{align*}
so for $t \geq C' k$ with a sufficiently large constant $C'$,  $S_1 \rightarrow 0$ as $k \rightarrow \infty$. 

On the other hand, for $k / \log k < l \leq k$, choose $t \geq \tC k$ with $\tC > 0$ large enough so that $\tgamma := \tlambda^{-1} (1+m_0 m_1 \gamma^{\tC})^{C} < 1$. Then as $t' \geq t / l \geq \tC$ we can bound 
\begin{align*}
	\sum_{k / \log k < l \leq k} \tlambda^{-l} \left( (1 + m_0 m_1 \gamma^{t'})^{l_{1}} - 1 \right) &\leq k \max_{k / \log k < l \leq k} \tlambda^{-l}  \left( (1 + m_0 m_1 \gamma^{\tC})^{Cl} -1 \right) \\
	&\leq k \max_{k / \log k < l \leq k} \left(\tlambda^{-1} (1+ m_0 m_1 \gamma^{\tC})^C \right)^{l} \\
	&= O(\tgamma^{k / \log k}).
\end{align*}

All in all, we have proved there exists a constant $C'$ such that for large enough $k$ and $t = C' k$, 
\[ 
	\sum_{n \in N_k} \TV{P_{n}^{t+d} - U_n} \leq S \leq o(\tlambda^{k}).
\]
As a consequence the number of integers $n$ in the interval $N_k$ for which the random walk has not mixed by time $t+d$ is $o(\tlambda^{k})$. 

On the other hand, for every integer $a \geq 1$, the number of multiples of $a$ in the interval $[c \tlambda^{k-1}, c \tlambda^{k}) $ is of order $O(\tlambda^{k-1})$. Consequently it must contain $\Omega(\tlambda^{k})$ integers coprime with all $a_i$, hence $\abs{N_k} \geq \Omega(\tlambda^{k})$. Ultimately this shows that the fraction of integers in $N_k$ for which the random walk has mixed at time $C' k$ can be made arbitrarily large by choosing $k$ large enough. 

\section{Markov partitions}\label{section:markov}

In this section we define Markov partitions and prove Propositions \ref{prop:existence_partitions}, \ref{prop:symbolic_twosided} and \ref{prop:symbolic_periodic}. The first use of Markov constructions for toral automorphisms goes back to the work of Berg \cite{berg1967conjugacy} and Adler and Weiss \cite{adler1967entropy} for dimension 2 . Since, many other constructions, with different degrees of explicitness, have been proposed for dynamical systems more general than automorphisms of the torus. This section is based on the general construction for Axiom A diffeomorphisms by Bowen \cite{bowen2008equilibrium}; see also \cite{benoist2002systemes}. For the reader interested in more explicit partitions, relating to arithmetic properties of the map $A$, see \cite{sidorov2003arithmetic, vershik1992arithmetic}. We note that these partitions are related to the $\beta$-ary expansions used in \cite{klyachko2020random}, but do not apply in the same generality as the Bowen construction.

\subsection{Hyperbolic dynamics}

We identify $A$ with the induced map on $\bT^{d}$. Recall $p$ denotes the natural projection $\bR^{d} \rightarrow \bT^{d}$, and that $d$ is the quotient metric induced by a norm $\norm{\cdot}$ on $\bR^{d}$ adapted to $A$.

For $x \in \bT^{d}$ and $\epsilon > 0$ define
\begin{equation}\label{eq:W_stable}
	\begin{split}
		&W^{s}(x) = \left\{ y \in \bT^{d}, \ d(A^n x, A^n y) \rightarrow 0 \right\} \\
		&\Wseps(x) = \left\{ y \in \bT^{d}, \ d(A^n x, A^n y) \leq \epsilon \:\forall n \geq 0 \right\} \\
		&W^{u}(x) = \left\{ y \in \bT^{d}, \ d(A^{-n} x, A^{-n} y) \rightarrow 0 \right\} \\
		&\Wueps(x) = \left\{ y \in \bT^{d}, \ d(A^{-n} x, A^{-n} y) \leq \epsilon \: \forall n \geq 0 \right\}.
	\end{split}
\end{equation}

$W^{s}(x)$, resp. $W^{u}(x)$ are the stable, resp. unstable manifold going through $x$. $\Wseps(x)$, resp. $\Wueps(x)$ are the local stable, resp. unstable manifold going through $x$. 

They can be described as: $W^s(x) = p(x + E_s)$, $W^{u}(x) = p(x + E_u)$ and $\Wseps(x) = p(B(x, \epsilon) \cap (x + E_s))$, $\Wueps(x) = p(B(x, \epsilon) \cap (x + E_u))$ where we recall $B(x, \epsilon)$ denotes the ball of radius $\epsilon$ and center $x$ in $\bR^{d}$.
The previous definitions readily imply the following.

\begin{lem}\label{lem:local_product}
	For all sufficiently small $\epsilon > 0$ such that $p$ is injective on $B(0, \epsilon)$, if $d(x,y) \leq \epsilon$ then $\Wseps(x) \cap \Wueps(y)$ consists of a single point, denoted $[x,y]$. The map $(x,y) \mapsto [x,y]$ is continuous.
\end{lem}

The construction of Markov partitions that we will present is quite general, for it is essentially based on compactness arguments combined with the following property, called shadowing of orbits. 

\begin{Def}
	Given $\alpha > 0$, a sequence $(x_k)_{k \in \bZ}$ is called an $\alpha$-pseudo-orbit if for all $k \in \bZ$, $d(A x_k, x_{k+1}) < \alpha$.
	Given $\beta > 0$, a point $x \in \bT^{d}$ $\beta-$ shadows the pseudo-orbit $(x_k)_{k \in \bZ}$ if for all $k \in \bZ$, $d(x_k, A^{k} x) < \beta$.
\end{Def}

\begin{prop}\label{prop:shadowing}
	For all $\beta > 0$ small enough, there exists $\alpha > 0$ such that every $\alpha$-pseudo-orbit $(x_k)_k$ is $\beta$-shadowed by a unique point $x$ of $\bT^{d}$.
\end{prop}

\begin{proof}
	Let $\beta > 0$. Uniqueness is provided by expansiveness (Proposition \ref{prop:expansiveness}), provided $\beta \leq \ec$. For the existence, it suffices to prove the analog result in $\bR^{d}$: if $(x_k)_{k \in \bZ}$ an $\alpha$-pseudo-orbit in $\bR^{d}$, then it can be $\beta$-shadowed by a true orbit, provided $\alpha$ is small enough. Indeed, lifting a pseudo-orbit of the torus to $\bR^{d}$, then we can project the shadowing orbit back to $\bT^{d}$ to prove the statement for the torus.
	
	Decompose $x_k = s_{k} + u_{k}$ with $s_{k} \in E_s, u_{k} \in E_u$ for all $k \in \bZ$. Since the metric is adapted to $A$, it is easily seen that the sequences $(s_k)_{k \in \bZ}$ and $(u_k)_{k \in \bZ}$ are both $\alpha$-pseudo-orbits. Write $S = A|_{E_s}$ and $T = A^{-1}|_{E_u}$. 
	Since $(s_k)_{k \in \bZ}$ is an $\alpha$-pseudo-orbit
	\[ 
		\norm{S^{k+1} s_{-(k+1)} - S^{k} s_{-k}} \leq \norm{S^{k}} \alpha \leq \lambda^{k} \alpha
	\]
	hence the sequence $(S^{k} s_{-k})_{k \in \bN}$ is a Cauchy sequence and we can define $s := \lim_{k \rightarrow \infty} S^{k} x_{-k}$. Then for all $m \geq 1 - k$,
	\begin{align*}
		\norm{s_k - S^{k} s} &\leq \sum_{i=0}^{k+m-1} \norm{S^{i} s_{k-i} - S^{i+1} s_{k-i-1}} + \norm{S^{k+m}(s_{-m}) - S^{k} s} \\
		&\leq \sum_{i=0}^{k+m-1} \lambda^{i} \alpha + \lambda^{k} \norm{S^{m} s_{-m} - s} \\
		&\leq \frac{\alpha}{1-\lambda}.
	\end{align*}
	by taking $m \rightarrow \infty$. 
	
	Repeating the argument with $T$, the sequence $(T^{k} u_{-k})_{k \in \bN}$ converges to a point $t \in \bR^{d}$ so that $\norm{u_k - T^{k} u} \leq \alpha /(1- \lambda)$. Since the metric is adapted,  $x := s + u$ $\beta$-shadows the pseudo orbit $(x_k)_{k \in \bZ}$ for $\beta = \alpha /(1- \lambda)$.
\end{proof}

\subsection{Markov Partitions}

Let $\epsilon > 0$ be sufficiently small so that the conclusion of Lemma \ref{lem:local_product} holds.

\begin{Def}
	A set $R \subset \bT^{d}$ is called a rectangle if it has diameter at most $\epsilon$ and, for all $x,y \in R$, $[x,y] \in R$. 
	A rectangle is said to be proper if $R = \overline{\mathrm{int} (R)}$. 
\end{Def}

Given a rectangle $R$ and $x \in R$, let
\begin{align*}
	W^s(x,R) := \Wseps(x) \cap R \qquad W^{u}(x,R) := \Wueps(x) \cap R .
\end{align*}

\begin{Def}
	A Markov partition is a finite covering $\cR = \{ R_1, \ldots, R_m \}$ of $\bT^{d}$ by proper rectangles such that 
	\begin{enumerate}[label=(\roman*)]
		\item $\mathrm{int} (R_i) \cap \mathrm{int} (R_j) = \emptyset$ for all $i \neq j$
		\item for all $x \in \mathrm{int}(R_i)$ if $Ax \in \mathrm{int}(R_j)$, then
		\begin{equation}
			A \left( W^s(x,R_i) \right) \subseteq W^s(Ax, R_j) \quad \text{and} \quad A \left( W^u(x,R_i) \right) \supseteq W^u(Ax, R_j).
		\end{equation}
	\end{enumerate}
\end{Def}

Notice that the second part of (ii) above is equivalent to 
\[ 
A^{-1} \left( W^u(x,R_i) \right) \subseteq W^u(A^{-1}(x), R_j).
\]
It is thus a dual statement of the first part obtained by replacing $A$ with $A^{-1}$, which has the effect of exchanging stable and unstable directions. 

Condition (ii) is really the Markov property of the partition. It ensures that if $\mathrm{int} R_i \cap A^{-1} \mathrm{int} R_j \neq \emptyset$ and $\mathrm{int} R_j \cap A^{-1} \mathrm{int} R_k \neq \emptyset$ then $\mathrm{int} R_i \cap A^{-1} \mathrm{int} R_j \cap \mathrm{int} A^{-2} R_k \neq \emptyset$: if $x, y$ are respectively in the first and second intersection, consider $A^{-1}[y,Ax]$.

We have all the tools for the proofs. We give all the essential arguments, but only sketch the technical details. We refer to \cite{bowen2008equilibrium} for a detailed proof, which extends to the general case of Axiom A diffeomorphisms.

\begin{proof}[Proof of Proposition \ref{prop:symbolic_periodic}]
	
	(i) Let $R$ be a rectangle and consider the compact set $U := \{ \omega \in \Omega \ | \ \omega_0 = R\}$. Since rectangles of a Markov partition are proper and of disjoint interior, one has $\pi^{-1}(\mathrm{int} R)  \subset U$. Because $\pi$ is surjective we deduce $\mathrm{int} R \subset \pi(U)$. Finally continuity implies $\pi(U)$ is compact, hence $R \subset \pi(U)$.
	
	(ii) Recall that $\pi$ is finite-to-one. Thus if $x$ is periodic of period $k$, the shift $\theta$ induces a permutation of the finite set of symbolic representations $\pi^{-1} \left(\{x, Ax, \ldots A^{k-1}x\}\right)$. For the integer $m$ such that the permutation induced by $\theta^m$ is the identity, one gets exactly that all sequences in $\pi^{-1}(x)$ are $m$-periodic.
\end{proof}

\begin{proof}[Proof of Proposition \ref{prop:symbolic_twosided}]
	Given two non-empty rectangles $R, S$ of a Markov partition, $S$ will be called a stable, resp. unstable subrectangle of $R$ if $S \subset R$, $S$ is proper and for all $x \in S$, $W^{s}(x,S) = W^{s}(x,R)$, resp. $W^{u}(x,S) = W^{u}(x,R)$.
	
	Consider a Markov partition $\cR = \{R_{1}, \ldots, R_m \}$, $\cA = \cA(R)$ and $\Omega := \Omega_{\cA}$ the associated shift.
	
	If $S \subseteq R_i$ is a non-empty unstable subrectangle and $\cA_{ij} = 1$, Lemma 3.17 in \cite{bowen2008equilibrium} establishes that $A(S) \cap R_j$ is a non-empty unstable rectangle of $R_j$. Similarly, $A^{-1}(S) \cap R_{j}$ is a non-empty stable rectangle of $R_i$.
	
	Let $\omega \in \Omega$. By the previous statement, $R_{\omega_0} \cap A^{-1}( R_{\omega_1})$ is a non-empty stable subrectangle of $R_{\omega_{0}}$. By an immediate induction, every intersection $\bigcap_{k=0}^{l} A^{-k} R_{\omega_k}$ is a non-empty stable subrectangle of $R_{\omega_0}$. By compactness, $\bigcap_{k=0}^{\infty} A^{-k} R_{\omega_k}$ is then also a non-empty stable subrectangle. Then reiterate the argument with $A^{-1}$ to deduce that $\pi(\omega) := \bigcap_{k \in \bZ} A^{-k} R_{\omega_k}$ is non-empty. By construction, the orbits of two points in this intersection must remain at distance $\epsilon$ from each other. If $\epsilon$ is taken smaller than the expansiveness constant, the intersection is thus a singleton.
	
	Let $\pi (\omega)$ be the unique point of this singleton. From the construction, it is immediate that $\pi \circ \theta = A \circ \pi$. By what precedes the diameter of the intersection $\bigcap_{k=-l}^{l} A^{-k} R_{\omega_k}$ tends to $0$ as $l$ tends to infinity, which implies the continuity of the map $\pi$. 
	
	We prove surjectivity. Let $\partial \cR := \bigcup_{R \in \cR} \partial R$ denote the union of all boundaries of rectangles. Its complement $\mathrm{int} (\cR)$ is an open dense subset of $\bT^{d}$. If $x \in \bigcap_{k} A^{k} \mathrm{int}(\cR)$, $x$ is the image of any sequence $\omega \in \Omega$ such that $A^{k} x \in R_{\omega_k}$ for all $k$. By Baire's theorem, an intersection of open dense subsets is dense, hence the image of $\pi$ contains an open dense subset of $\bT^{d}$. However $\pi(\Omega)$ has to be compact by continuity, whence $\pi(\Omega) = \bT^{d}$.
	
	
	Finally, let us prove that any point of $\bT^{d}$ cannot have more than $m^{2}$ pre-images under $\pi$.  Suppose there exist $\omega^{(1)}, \ldots, \omega^{(m^{2}+1)} \in \Omega$ with the same image $x$. Then one can find $k_1 \leq k_2$ such that the indices between $k_1$ and $k_2$ of these sequences do not coincide. However there is at most $m^{2}$ possibilities when choosing the $k_1$-th and $k_2$-th index, therefore there exists among these sequences a pair $(\omega, \omega')$ such that $\omega_{k_1}=\omega'_{k_1}$, $\omega_{k_2} = \omega'_{k_2}$ and $\omega_{l} \neq \omega'_{l}$ for $k_1 < l < k_2$. 
	
	Now from the Markov property of the partition, the intersection $\bigcap_{k=k_1}^{k_2} A^{-k} \mathrm{int} R_{\omega_k}$ must be non-empty. Take $u$ in this open set and consider $\tau \in \pi^{-1}(u)$. Since $u$ is an interior point $\tau$ and $\omega$ coincide on indexes between $k_1$ and $k_2$. 
	
	On the other hand, since indexes $k_1$ and $k_2$ of $\omega$ and $\omega'$ coincide, one can build a new sequence $\tau'$ by replacing the block of $\tau$ between $k_1$ and $k_2$ by that of $\omega'$, that is we define
	\[ 
	\tau'_ k:=  \left\{ \begin{array}{l l}
		\tau_k & \text{for $k \leq k_1$ or $k \geq k_2$} \\
		\omega'_{k} & \text{if $k_1 \leq k \leq k_2$}
	\end{array} \right. .
	\]
Consider finally the point $v := \pi(\tau')$. By construction $A^{k} u$ and $A^{k} v$ are contained in the same rectangles for all $k \leq k_1$ and $k \geq k_2$, whereas for $k_1 \leq k \leq k_2$, they are contained in the rectangles $R_{\omega_k}$ and $R_{\omega'_{k}}$ respectively. However the latter both contain the point $A^{k} x$, so $A^{k} u ,A^{k}v$ are at distance at most $2 \eta$ from each other by triangle inequality. Thus $d(A^{k}u, A^{k}v) \leq 2 \eta$ for all $k$. Consequently if $2 \eta \leq \ec$ expansiveness ensures that $A^{l} u \in W^{s}(A^{l},R_{\omega_{l}})$. In the end we obtain $A^{l} u \in \mathrm{int} R_{\omega_l} \cap R_{\omega'_{l}}$, which contradicts the fact that rectangles have disjoint interior.
\end{proof}

\begin{proof}[Proof of Proposition \ref{prop:existence_partitions}]
	Let $\beta > 0$ and consider $\eta > 0 $ such that all $2 \eta (1 + \norm{A})$-pseudo-orbits are $\beta$-shadowed by a true orbit (Proposition \ref{prop:shadowing}). Consider a finite cover of $\bT^{d}$ with balls $B(x_i, \eta)$, $i=1, \ldots, m$. Let $\cA$ be the matrix defined by $\cA_{ij} = 1$ if and only $A x_i \in B(x_j, 2 \eta (1+ \norm{A}))$ and consider the shift $\Omega := \Omega_{\cA}$ associated.
	
	For all $\omega \in \Omega$, the sequence $(x_{\omega_n})_{n}$ is a $2 \eta (1+\norm{A})$-pseudo-orbit and thus can be $\beta$-shadowed by points of a set $\pi(\omega) \subset \bT^{d}$ (which is a singleton in the invertible case). 
	
	By construction, $A \circ \pi = \pi \circ \theta$. $\pi$ is surjective: every point $x \in \bT^{d}$ is an image $\pi(\omega)$ for any sequence $\omega$ such that $A^{n}x \in B(x_{\omega_n}, \eta)$ for all $n \in \bZ$. These sequences are indeed in $\Omega$, since
	\[ 
	d(A x_{\omega_n}, x_{\omega_{n+1}}) \leq d(A x_{\omega_n}, A^{n+1} x) + d(A^{n+1} x, x_{\omega_{n+1}}) \leq (\norm{A} + 1) \eta < 2 (1+ \norm{A}) \eta.
	\]
	$\pi$ is also continuous when giving $\Omega$ the product topology. The argument is similar to the one given in the previous proof. 
	
	For $i \in [m]$, let 
	\[ 
	T_i := \pi \left( \left\{ \omega \in \Omega, \omega_0 = i \right\} \right).
	\]
	By construction for all $\omega \in \Omega$ $d(\rho(\omega), x_{\omega_0}) \leq \beta$, so the sets $T_i$ have diameter at most $2 \beta$, which can be made arbitrarily small. From the continuity of $\pi$, the sets $T_i$ are closed sets. 
	
	Consider now the local product on $\Omega$ defined for $\omega, \omega' \in \Omega$ with $\omega_0 = \omega'_0$ by
	\[ 
	[\omega, \omega']_k = \left\{ \begin{array}{l l}
		\omega_k & \text{if $k \geq 0$} \\
		\omega'_k & \text{if $k \leq 0$}
	\end{array} \right.
	\]
	Then $d(A^k \pi(\omega), A^k \pi([\omega, \omega'])) \leq 2 \beta$ for all $k \geq 0$, and $d(A^k \pi(\omega), A^k \pi([\omega, \omega'])) \leq 2 \beta$ for all $k \leq 0$. Thus $\pi([\omega, \omega']) \in W^s_{2\beta}(\pi(\omega)) \cap W^u_{2\beta}(\pi(\omega'))$, so that $\pi([\omega, \omega']) = [\pi(\omega), \pi(\omega')]$ by Lemma \ref{lem:local_product}. This immediately implies that the sets $T_i$ are rectangles.
	
	Suppose then that $x = \pi(\omega)$, with $\omega_0 = i, \omega_1 = j$. From $\pi \circ \theta = A \circ \pi$ we deduce $W^{s}(x,T_i) = \pi \left( \{ \omega' \in \Omega, \forall k \geq 0,  \omega_k = \omega'_k \}\right)$. Thus 
	\[ 
	A W^{s}(x, T_i) = \pi \left( \{ \omega' \in \Omega, \forall n \geq -1 \ \omega'_{n} = \omega_{n+1} \}\right) \subset \pi \left( \{ \omega' \in \Omega, \forall n \geq 0 \  \ \omega'_{n} = \omega_{n+1} \}\right) = W^{s}(Ax,T_j),
	\]
	hence 
	\[ 
	A(W^{s}(x, T_i)) \subseteq W^{s}(Ax, T_j).
	\]
	Using $A^{-1}$ in place of $A$, we can prove similarly $A^{-1}(W^{u}(x, T_i)) \subseteq W^{u}(A^{-1}x, T_j)$. Therefore the rectangles $T_i$ satisfy almost all the properties of a Markov partition, except there may not be proper or of disjoint interiors. 
	
	To get rectangles satisfying these properties, the idea is to divide the $T_i$ further into smaller rectangles. Namely if $T_j \cap T_k \neq \emptyset$, let
	\begin{align*}
		&T^{1}_{j,k} = \left\{x \in T_j, W^{u}(x,T_j) \cap T_k \neq \emptyset, W^{s}(x,T_j) \cap T_k \neq \emptyset \right\} = T_j \cap T_k \\
		&T^{2}_{j,k} = \left\{x \in T_j, W^{u}(x,T_j) \cap T_k \neq \emptyset, W^{s}(x,T_j) \cap T_k = \emptyset \right\} \\
		&T^{3}_{j,k} = \left\{x \in T_j, W^{u}(x,T_j) \cap T_k = \emptyset, W^{s}(x,T_j) \cap T_k \neq \emptyset \right\} \\
		&T^{4}_{j,k} = \left\{x \in T_j, W^{u}(x,T_j) \cap T_k = \emptyset, W^{s}(x,T_j) \cap T_k = \emptyset \right\}
	\end{align*}
	Then for all $x,y \in T_j$, the fact that $W^{s}([x,y], T_j) = W^{s}(x, T_j)$ and $W^{u}([x,y], T_j) = W^{u}(y, T_j)$ proves that the $T^{i}_{j,k}$ are rectangles. For any $x \in \bT^{d}$ let 
	\[ 
		R(x) := \bigcap \left\{ \mathrm{int} (T^{i}_{j,k}): T_j \cap T_k \neq \emptyset \text{ and } x \in T^{i}_{j,k} \right\}
	\]
	and $\cR := \{ \overline{R(x)}, x \in \bT^{d} \}$. It can be easily checked that $R(x) \cap R(x') \neq \emptyset$ implies $R(x) = R(x')$ so the set $\cR$ is finite. Therefore $\cR$ defines a finite covering by proper rectangles of disjoint interiors. It remains to prove the second property of Markov partitions.
	
	Suppose $y \in W^{s}(x, \overline{R(x)})$. Suppose that $x \in \mathrm{int} (T_i)$ and $Ax \in \mathrm{int} (T_j)$. From the property proved previously for the $T_i$, $Ay \in W^{s}(Ax, T_j)$. We will argue by contradiction to prove that $R(Ax) = R(Ay)$. If this does not hold, it implies the existence of a rectangle $T_k$ intersecting $T_j$ such that $Ax, Ay$ are not in the same $T^{\cdot}_{j,k}$. From the definition of $T_{j,k}^{\cdot}$ we can suppose without loss of generality that $W^{u}(Ax, T_j) \cap T_k \neq \emptyset$ and $W^{u}(Ay, T_j) \cap T_k = \emptyset$. Since $A W^{u}(x, T_i) \supset W^{u}(Ax, T_j)$, we can thus find $z \in W^{u}(x, T_i)$ such that $Az \in T_j \cap T_k = T^{1}_{j,k}$. Consider $l$ such that $z \in T_{i} \cap T_l = T^{1}_{i,l}$. Since $R(x) = R(y)$, $x,y$ are in the same subrectangle $T^{\cdot}_{i,l}$, so there exists a point $z' \in W^{u}(y, T_i) \cap T_l$. Then $z'' := [z,z'] = [z,y]$ is in $W^{s}(z,T_l) \cap W^{u}(y,T_i)$ and $Az'' = [Az, Ay] \in W^{s}(Az, T_k) \cap W^{u}(Ay, T_j)$, which contradicts $W^{u}(Ay, T_j) \cap T_k = \emptyset$. Hence $R(Ax) = R(Ay)$ which implies $A W^{s}(x,\overline{R(x)}) \subset W^{s}(Ax, \overline{R(Ax)})$.

\end{proof}

\subsection*{Acknowledgements}
We thank Ioannis Iakovoglou for useful discussions. We also thank the anonymous referees for pointing out a gap in a proof in the first version of this manuscript and for reference \cite{klyachko2020random}.

\bibliographystyle{plainurl}
\bibliography{accelerating_hyperbolic.bib}

\end{document}